\documentclass[11pt]{amsart}
\usepackage{amsmath,amssymb,amsthm}
\usepackage[latin1]{inputenc}
\usepackage{mathrsfs}
\usepackage{version, tabularx, multicol}
\usepackage{graphicx,float,psfrag}
\usepackage{stmaryrd}
\usepackage{color,latexsym,amsfonts}

\newcommand{\reff}[1]{(\ref{#1})}

\theoremstyle{plain}
\newtheorem{theo}{Theorem}[section]
\newtheorem{theo*}{Theorem}
\newtheorem{theoA}{Theorem}

\newtheorem{cor}[theo]{Corollary}
\newtheorem{prop}[theo]{Proposition}
\newtheorem{lem}[theo]{Lemma}
\newtheorem{defi}[theo]{Definition}
\theoremstyle{remark}
\newtheorem{rem}[theo]{Remark}

\newcommand{\ca}{{\mathcal A}}

\newcommand{\cc}{{\mathcal C}}

\newcommand{\ce}{{\mathcal E}}

\newcommand{\cf}{{\mathcal F}}
\newcommand{\cg}{{\mathcal G}}
\newcommand{\ch}{{\mathcal H}}
\newcommand{\ci}{{\mathcal I}}

\newcommand{\cl}{{\mathcal L}}
\newcommand{\cn}{{\mathcal N}}
\newcommand{\cm}{{\mathcal M}}

\newcommand{\cs}{{\mathcal S}}

\newcommand{\cw}{{\mathcal W}}
\newcommand{\cx}{{\mathcal X}}

\newcommand{\cz}{{\mathcal Z}}

\newcommand{\ed}{E_\mathrm{d}}
\newcommand{\eg}{E_\mathrm{g}}
\newcommand{\eed}{e_\mathrm{d}}
\newcommand{\eeg}{e_\mathrm{g}}
\newcommand{\pd}{p_\mathrm{d}}
\newcommand{\pg}{p_\mathrm{g}}

\newcommand{\E}{{\mathbb E}}

\newcommand{\N}{{\mathbb N}}
\renewcommand{\P}{{\mathbb P}}

\newcommand{\R}{{\mathbb R}}

\newcommand{\T}{{\mathbb T}}

\newcommand{\rZ}{{\mathbf Z}}

\newcommand{\bt}{{\mathbf t}}
\newcommand{\bff}{{\mathbf f}}

\newcommand{\ind}{{\bf 1}}

\newcommand{\clo}{{\rm cl}\;}

\newcommand{\Card}{{\rm Card}\;}

\newcommand{\inv}[1]{\mathop{\frac{1}{ #1}}\nolimits}
\newcommand{\expp}[1]{\mathop {\mathrm{e}^{ #1}}}

\newcommand{\sgn}{{\rm sgn}}

\newcommand{\lb}{[\![}
\newcommand{\rb}{]\!]}

\title[Simulation of a genealogical tree]{Exact simulation of
  the genealogical tree for a stationary branching population and
  application to the asymptotics of its total length}

\date{\today}

\author{Romain Abraham}
\address{Romain Abraham,
Institut Denis Poisson,
  Universit\'{e} d'Orl\'{e}ans,
  Universit\'e de Tours,
  CNRS, France}
\email{romain.abraham@univ-orleans.fr}

\author{Jean-Fran\c{c}ois Delmas}
\address{Jean-Fran\c{c}ois Delmas,
 Universit\'{e} Paris-Est, CERMICS (ENPC), France}
\email{delmas@cermics.enpc.fr}

\begin{document}



\begin{abstract}
  We consider a model of stationary population with random size given by
  a continuous state branching process with immigration with a quadratic
  branching mechanism.  We give an exact elementary simulation procedure
  of the genealogical tree of  $n$ individuals randomly chosen among the
  extant population at a given time.   Then, we prove the convergence of
  the renormalized total length of this genealogical tree as $n$ goes to
  infinity, see also Pfaffelhuber,  Wakolbinger and Weisshaupt (2011) in
  the context of  a constant size population. The  limit appears already
  in Bi and Delmas (2016) but with a different approximation of the full
  genealogical tree. The proof is based  on the ancestral process of the
  extant population  at a  fixed time  which was  defined by  Aldous and
  Popovic (2005) in the critical case.
\end{abstract}

\maketitle

\section{Introduction}

Continuous state branching (CB)  processes are stochastic processes that
can  be obtained  as the  scaling limits  of sequences  of Galton-Watson
processes when the initial number of individuals tends to infinity. They
hence  can be  seen as  a model  for a  large branching  population. The
genealogical structure of a CB process  can be described by a continuum
random tree introduced  first by Aldous \cite{a:crtI}  for the quadratic
critical  case, see  also  Le  Gall and  Le  Jan \cite{lglj:bplpep}  and
Duquesne and  Le Gall  \cite{dlg:rtlpsbp} for  the general  critical and
sub-critical cases. We shall only consider the quadratic case; it is
characterized by a 
branching mechanism $\psi_\theta$:
\begin{equation}\label{eq:psi}
\psi_\theta(\lambda)=\beta \lambda^ 2+ 2\beta\theta \lambda, \quad
\lambda\in [0, +\infty ),
\end{equation}
where $\beta>0$  and $\theta\in  \R$. The sub-critical  (resp. critical)
case  corresponds  to  $\theta>0$  (resp.   $\theta=0$).  The  parameter
$\beta$  can be  seen as  a time  scaling parameter,  and $\theta$  as a
population size parameter.  

In this  model  the  population  dies  out  a.s.~in  the  critical  and
sub-critical  cases.   In  order  to  model  branching  population  with
stationary size distribution,  which corresponds to what  is observed at
an ecological equilibrium, one can  simply condition a sub-critical or a
critical   CB  to   not  die   out.    This  gives   a  Q-process,   see
Roelly-Coppoleta     and      Rouault     \cite{rcr:pdwcfl},     Lambert
\cite{l:qsdcsbpcne} and Abraham and Delmas \cite{ad:wdlcrtseppnm}, which
can  also  be  viewed  as  a   CB  with  a  specific  immigration.   The
genealogical structure  of the Q-process  in the stationary regime  is a
tree with an  infinite spine.  This infinite spine has  to be removed if
one adopts the immigration point of  view, in this case the genealogical
structure  can be  seen  as  a forest  of  trees.   For $\theta>0$,  let
$Z=(Z_t, t\in  \R)$ be this  Q-process in  the stationary regime,  so that
$Z_t$ is the size  of the population at time $t\in \R$.   The process
$Z$ is a Feller
diffusion (see for example Section 7 in \cite{cd:spsmrcatsbp}), solution
of the SDE:
\begin{equation}
   \label{eq:SDE-Z}
dZ_t=\sqrt{2\beta Z_t}\, dB_t+2\beta(1-\theta Z_t)dt,
\end{equation}
where  $(B_t, t\geq  0)$  is a standard Brownian motion.   See  Chen and  Delmas
\cite{cd:spsmrcatsbp}  for  studies on  this  model  in a  more  general
framework. See Section  \ref{sec:local_times} for other contour processes
associated with the process $Z$. 
Let $A_t$  be the time to the most  recent common ancestor of
the population living  at time $t$, see \reff{eq:def-Ah2}  for a precise
definition.     According    to     \cite{cd:spsmrcatsbp},    we    have
$\E[Z_t]=1/\theta$,  and $\E[A_t]=3/4\beta\theta$,  so that  $\theta$ is
indeed   a   population  size   parameter   and   $\beta$  is   a   time
parameter.\medskip

Aldous  and Popovic  \cite{ap:cbpmb} (see  also Popovic  \cite{p:agcbp})
give a description of the genealogical  tree of the extant population at
a fixed  time using  the so-called  ancestral process  which is  a point
process  representation  of the height of the  branching points of a planar tree  in  a  setting very  close  to
$\theta=0$  in  the  present  model.   We  extend  the  presentation  of
\cite{ap:cbpmb}   to  the   case   $\theta\geq   0$,  which  can  be summarized  as
follows.  The  ancestral  process, see  Definition
\ref{def:anc},           is           a          point      measure
$\ca(du,  d\zeta)=\sum_{i\in \ci}  \delta_{u_i, \zeta_i}(du,d\zeta)$  on
$\R^*\times  (0,+\infty)$, where  $u_i$  represents  the (position of
the) individual $i$
in the extant population   and
$\zeta_i$ its ``age''. (The position $0$ will correspond to the position
of the immortal individual. The order on $\R$ provides a natural order on the
individuals through their positions, which means that we are dealing
with ordered or planar genealogical tree.)
From this ancestral process, we construct informally a genealogical tree
${\mathfrak{T}}(\ca)$ as follows. We view
this process as  a sequence of vertical segments in  $\R^2$, the tops of
the   segments  being   the  $u_i$'s   and  their   lengths  being   the
$\zeta_i$'s.  We add  the half  line $\{0\}\times  (-\infty,0]$ in  this
collection of segments.  We then attach the bottom of  each segment such
that $u_i>0$ (resp. $u_i<0$) to the first longer segment to the left
(resp.~ right) of it. See Figure \ref{fig:ancestral-intro} for an example.
This provides the planar tree
$\mathfrak{T}(\ca)$ associated with the ancestral process $\ca$ (see Proposition
\ref{proof:T(A)} for the definition and properties of this 
locally compact  real tree with a unique semi-infinite branch). 
 To state our result, we decompose the
extant population at time $t$ into two sub-populations so that its size $Z_t$ is distributed as 
$\eg+\ed$, where $\eg$ (resp. $\ed$) is the size of the population
grafted on the left (resp. on the right) of the infinite spine. 
We state the main result of Section \ref{sec:ancestral},  see   Propositions
\ref{prop:ancestral}  and \ref{prop:A=G}. 

\begin{theoA}
   \label{theoA:ancestral}
   Let  $\theta\geq 0$.   Let $\eg,\ed$  be independent  
   exponential random variables  with    mean
   $1/2\theta$, and  with   the  convention   that  $\ed=\eg=+\infty   $  if
   $\theta=0$.  Conditionally  given $(\eg,\ed)$, the  ancestral process
   $\ca(du, d\zeta)$ is a Poisson point measure with intensity:
\[
\ind_{(-\eg,\ed)}(u)\, du\, |c'_\theta(\zeta)|d\zeta, 
\]
with $c_\theta$ given by
\begin{equation}\label{eq:def-c}
\forall h>0,\quad c_\theta(h)=\begin{cases}
\frac{2\theta}{\expp{2\beta\theta h}-1} & \mbox{if }\theta>0,\\
(\beta h)^{-1} & \mbox{if }\theta=0.
\end{cases}
\end{equation}
Furthermore, the tree ${\mathfrak{T}}(\ca)$
is distributed as the genealogical tree of  the extant population at
a fixed time $t\in \R$. 
\end{theoA}

\begin{figure}[H]
\begin{center}
\includegraphics[width=7cm]{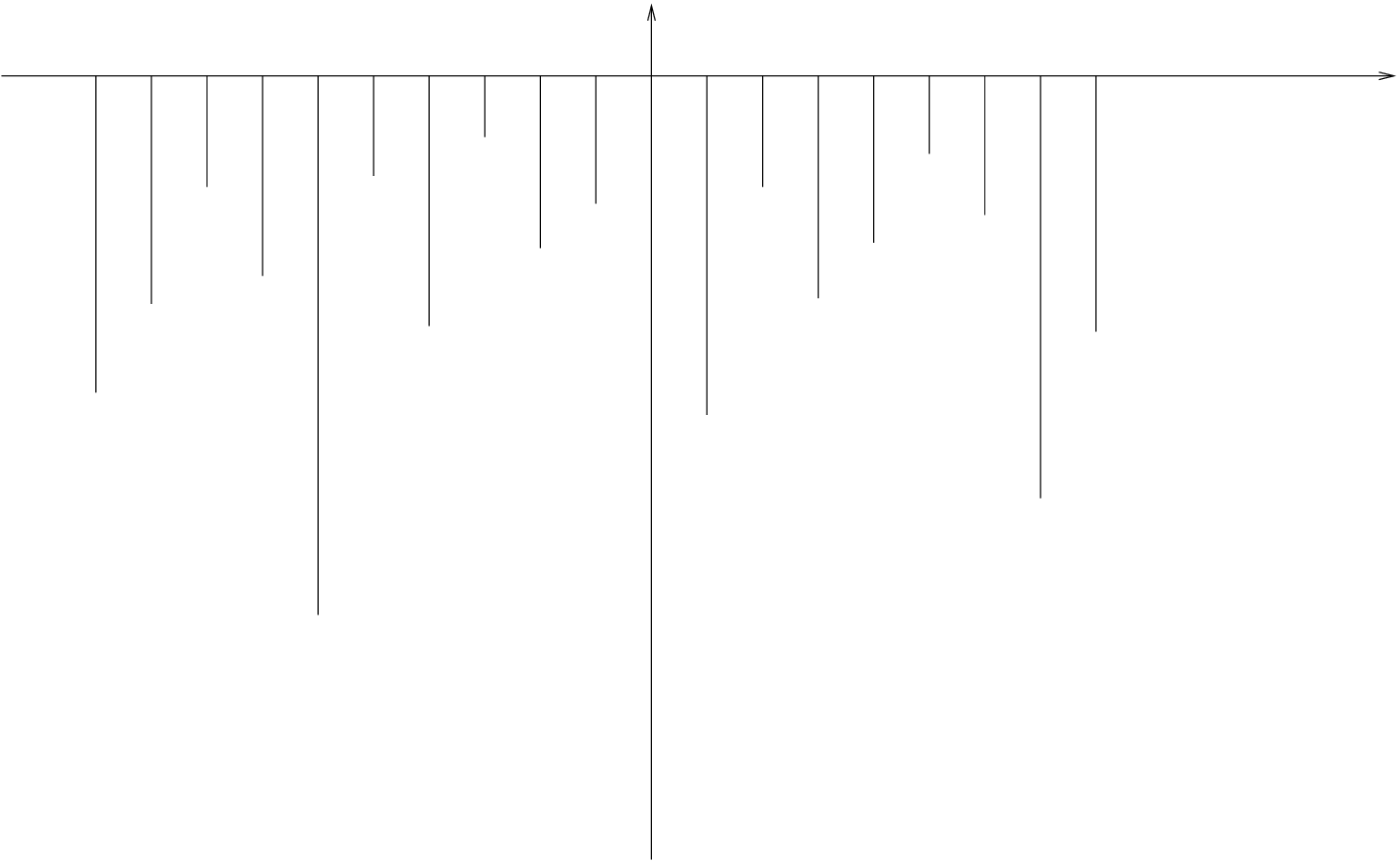}
\includegraphics[width=7cm]{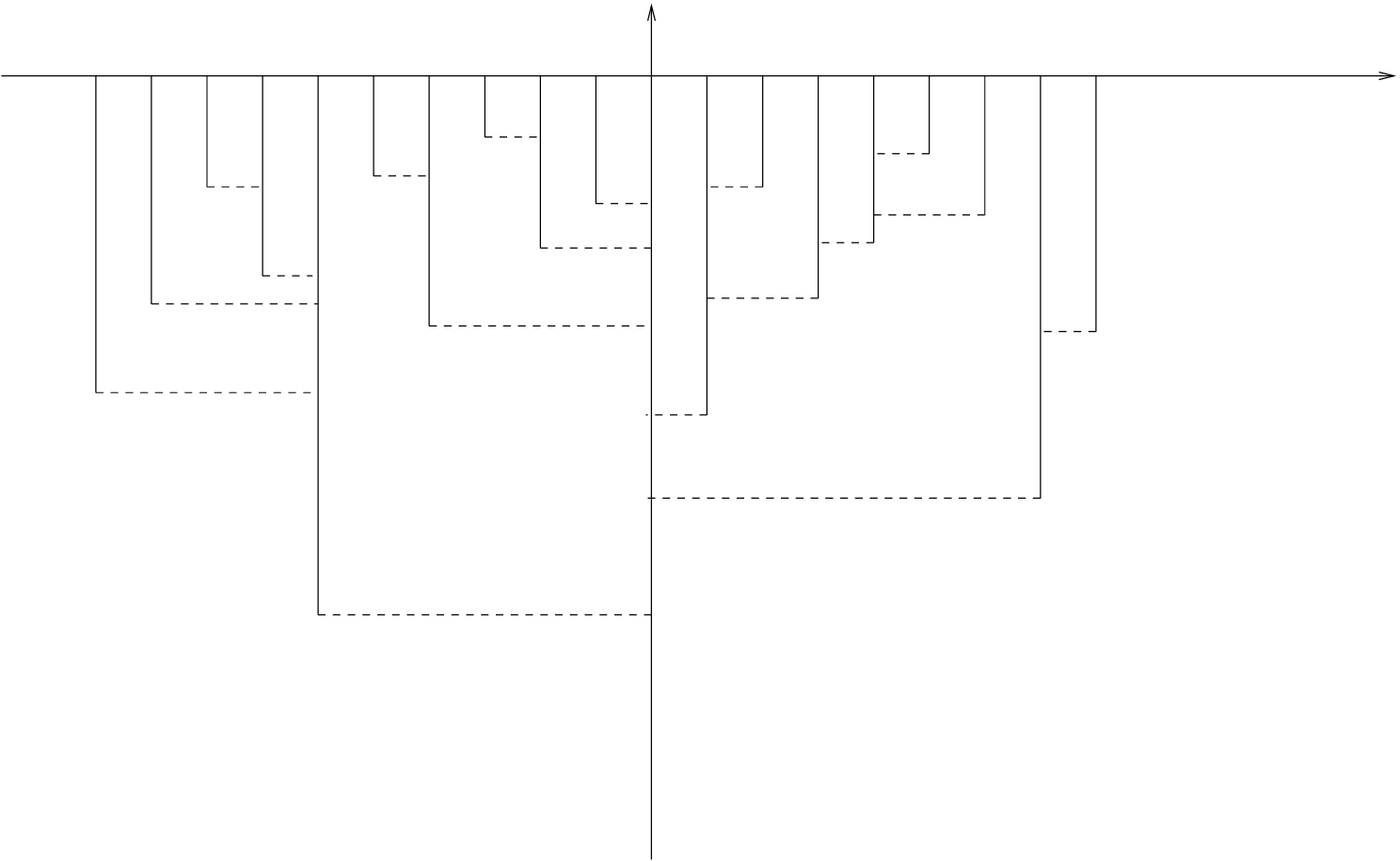}
\caption{An instance of an ancestral process and the corresponding
  genealogical tree}
\label{fig:ancestral-intro}
\end{center}
\end{figure}

The  ancestral  process  description  allows to  give  elementary  exact
simulations of the genealogical tree  of $n$ individuals randomly chosen
in the  extant population at time  0 (or at  some time $t\in \R$  as the
population  has  a stationary  distribution).  We  present here the  static
simulation for fixed $n\geq 2$ given in Subsection \ref{sec:static}, see
also Lemma \ref{lem:arbre-n} there. See
Figures   \ref{fig:sim1}  for  an 
illustration for $n=5$.

\begin{theoA}
\label{theo:simB}
  Let  $n\in  \N^*$. 
\begin{itemize}
\item[(i)] \textbf{Size of the extant population}.  Let $\eg, \ed$  be  independent  exponential random
  variables with mean $1/2\theta$.  ($\eg+ \ed$ corresponds to the size of the extant population.)
\item[(ii)] \textbf{Picking  $n$ individuals  in the extant
    population}. 
Let $(X_k, k\in \{1, \ldots, n\})$
  be, conditionally on $(\eg, \ed)$, 
  independent uniform random variables on $[-\eg, \ed]$ and set
  $X_0=0$. (The individual $0$ corresponds to the infinite spine.)
\item[(iii)] \textbf{The ``age'' of the individuals}. 
For $k\in \{1, \ldots, n\}$, set $\Delta_k$ as the length of the
intermediate interval to the next $X_j$ on the right if $X_k<0$ or on
the left if $X_k>0$:
\[
\Delta_k=
\begin{cases}
   X_k -\max\{X_j,\, X_j<X_k\text{ and } 0\leq j\leq n\} & \text{if $X_k>0$},\\
- X_k +\min\{X_j,\, X_j>X_k\text{ and } 0\leq j\leq n\} & \text{if $X_k>0$}.
\end{cases}
\]
Conditionally on  $(\eg, \ed,  X_1, \ldots,  X_n)$, let
     $(\zeta_{k}^{ \text{S}}, 1\leq  k\leq n)$  be independent  random variables
     such that   $\zeta_{k}^{ \text{S}}$ is  distributed as, with  $U$ is uniform on $[0, 1]$:
\[
\inv{2\theta \beta} \log\left(1- \frac{2\theta \Delta_k}{\log(U)}\right).
\]

\item[(iii)] \textbf{The tree}. 
Let $\mathfrak{T}_n^{ \text{S}}$ be the tree associated with the ancestral process
     $\sum_{k=1}^n \delta_{(X_k, \zeta_{k}^{ \text{S}})}$.

\end{itemize}
Then, the tree $\mathfrak{T}_n^{ \text{S}}$ is distributed as  the genealogical tree of $n$
  individuals picked  uniformly at random among the extant population. 
\end{theoA}

The  notion of genealogical tree is
appropriate for certain abstractions of genetic relations (mitochondrial
DNA  that is  maternally  inherited when  ignoring  paternal leakage  or
hetero-mitochondrial  inheritance) in  diploid  organisms. It is however 
unclear how to extend our exact simulation algorithm to pedigree-conditioned genealogies
as  formalised in Sainudiin,  Thatte, and  Véber \cite{stv:zrdp}.

In the spirit of Theorem \ref{theo:simB}, we  also provide   two  dynamic  simulations  in
Subsections   \ref{sec:dynamic1}  and   \ref{sec:dynamic2},  where   the
individuals  are taken  one by  one and  the genealogical  tree is  then
updated. Our framework allows also  to simulate the genealogical tree of
$n$ extant  individuals conditionally given  the time $A_0$ to  the most
recent    common    ancestor    of   the    extant    population,    see
Subsection~\ref{sec:cond_simul}.  Let us stress that the existence of an
elementary  simulation  method  is  new  in  the  setting  of  branching
processes  (in  particular because  this  method  avoid the  size-biased
effect on the population which usually comes from picking individuals at
random), and the question goes back to Lambert \cite{l:ctbp} and Theorem
4.7 in \cite{cd:spsmrcatsbp}.  \medskip

The  ancestral process  description  allows also  to  compute the  limit
distribution of the total length of  the genealogical tree of the extant
population at  time $t\in\R$.   More precisely,  let $\Lambda_n$  be the
total  length of  the tree  of $n$  individuals randomly  chosen in  the
extant  population at  time $t$,  see \reff{eq:Lambda'n}  for a  precise
definition.  We state the main result of Section \ref{sec:total-l}, 
see Theorem \ref{thm:cvLn}. 
\begin{theoA}
   \label{thmA:cvLn}
   The                                                          sequence
   $\left(\Lambda_n - \E[\Lambda_n|Z_t],  n\in \N ^*\right)$ converges
   a.s.~and in $L^2$ towards a limit, say $\cl_t$, as $n$ tends to $+\infty$. And we
   have:
\[
\E[\Lambda_n|Z_t]=\frac{Z_t}{\beta}  \log\left(\frac{n}{2\theta Z_t
     }\right) + O(n^{-1}\log(n)).
\] 
\end{theoA}
This result is in the spirit of Pfaffelhuber, Wakolbinger and Weisshaupt
\cite{pww:tlec} on the  tree length of the coalescent, which  is a model
for constant population size. The fact that the same shift in $\log(n)$
appears in \cite{pww:tlec} and in Theorem \ref{thmA:cvLn} comes from the
fact that  the speed of coming down
from infinity  (or the birth rate of new branches near the top of the
tree in forward time)
 is of the same order  for the  Kingman coalescent (see \cite{bbl})
 and this  model (see Corollary 6.5 and Remark 6.6 in
 \cite{cd:spsmrcatsbp}). 
\medskip

As part of Theorem \ref{thm:cvLn}, we also
get that $\cl_{t}$ coincides with the  limit of the shifted total length
$L_\varepsilon$ of  the genealogical tree  up to $t-\varepsilon$  of the
individuals alive  at time  $t$ obtained  in \cite{bd:tl}:  the sequence
$(L_\varepsilon   -  \E[L_\varepsilon|Z_t],   \varepsilon>0)$  converges
a.s.~towards  $\cl_{t}$  as  $\varepsilon$   goes  down  to  zero.   See
\cite{bd:tl} for some properties of the process $(\cl_{t}, t\in \R)$
such as  the Laplace transform  of  $\cl_{t}$ which is 
given by, for $\lambda>0$:
\begin{equation}
   \label{eq:def-Z-f}
\E\left[\expp{-\lambda \cl_{t}}|Z_t\right]
=\expp{2\theta Z_t \, \varphi(\lambda/(2\beta\theta))}, 
\quad\text{with}\quad
\varphi(\lambda )=\lambda \int_0^1 \frac{1- v^\lambda}{1-v} \, dv.
\end{equation}
The proof  of Theorem \ref{thmA:cvLn} is based on technical  $L^2$ computations. \medskip

The  paper is  organized  as  follows.  We  first  introduce in  Section
\ref{sec:notations}  the  framework of  real  trees  and we  define  the
Brownian CRT  that describes the  genealogy of  the CB in  the quadratic
case. Section  \ref{sec:ancestral} is devoted  to the description  via a
Poisson point measure of the  ancestral process of the extant population
at time 0 and Section \ref{sec:simul} gives the different simulations of
the  genealogical  tree  of  $n$ individuals  randomly  chosen  in  this
population.  Then,  Section  \ref{sec:total-l} concerns  the  asymptotic
length of the genealogical tree for those $n$ sampled individuals.

\section{Notations}\label{sec:notations}
We set $\R^*=(-\infty , 0) \cup (0, +\infty )$,  $\N^*=\{1, 2,
\ldots, \}$ and $\N=\N^* \cup \{0\}$. Usually $I$ will denote generic
index set which might be finite, countable or uncountable. 

\subsection{Excursion measure for Brownian motion with drift}
\label{sec:em-bmd}
In this section we state some well-known results on 
excursion measures of the Brownian motion with drift. Let $B=(B_t, t\geq
0)$  a standard Brownian motion and let $\beta>0$ be fixed. Let
$\theta\in \R$. We consider 
$B^{(\theta)}=(B^{(\theta)}_t,  t\geq 0)$  a Brownian motion  with
drift $-2\theta$ and scale $\sqrt{2/\beta}$: 
\begin{equation}\label{eq:def-Btheta}
B^{(\theta)}_t =\sqrt{\frac{2}{\beta}}\, B_t-2\theta t, \quad t\geq 0.
\end{equation}
Consider the minimum process $I^{(\theta)}=(I_t^{(\theta)}, t\geq 0)$ of
$B^{(\theta)}$                         defined                        by
$I^{(\theta)}_t=\min_{u\in     [0,     t]}     B^{(\theta)}_u$.      Let
$n^{(\theta)}(de)$   be   the   excursion   measure   of   the   process
$B^{(\theta)} - I^{(\theta)}$ above $0$  associated with its local time at
$0$ given by  $ -\beta I^{(\theta)}$. This normalization  agrees with the
one  in \cite{dlg:rtlpsbp}  given  for $\theta\geq  0$,  see the  remark
below.    Let    $\sigma=\sigma(e)=\inf\{s>0,     \,    e(s)=0\}$    and
$\zeta=\zeta(e)=\max_{s\in  [0, \sigma]}(e_s)$  be  the  length and  the
maximum of the excursion $e$.

\begin{rem}
   \label{rem:norm-dlg}
   In  this  remark, we  assume that  $\theta>0$ (i.e. the Brownian motion has a negative drift).   In  the
   framework   of   \cite{dlg:rtlpsbp},   see   Section   1.2   therein,
   $B^{(\theta)}$  is  the  height  process  which  codes  the  Brownian
   continuum random  tree (CRT)  with branching  mechanism $\psi_\theta$
   defined by \reff{eq:psi}.  It is obtained from  the underlying L\'evy
   process $X=(X_t, t\geq 0)$, which  in the case of quadratic branching
   mechanism     is     the      Brownian     motion     with     drift:
   $X_t=\beta  B^{(\theta)}_t=\sqrt{2\beta}\,  B_t-2\beta\theta t$  (see
   formula     (1.7)    in     \cite{dlg:rtlpsbp}).     According     to
   \cite{dlg:rtlpsbp}  Section 1.1.2,  considering  the minimum  process
   $I=(I_t, t\geq  0)$, with $I_t=\min_{u\in  [0, t]} X_u$,  the authors
   choose the normalization in such a way that $-I$ is the local time at
   0 of $X-I$.  The choice of the normalization of the local time at $0$
   of  $B^{(\theta)}  - I^{(\theta)}$  is  justified  by the  fact  that
   $I=\beta I^{(\theta)}$.  Recall the definition of $c_\theta$ in \eqref{eq:def-c}.
   Then from Section 3.2.2  and Corollary 1.4.2
   in \cite{dlg:rtlpsbp}, we have that for $\theta\geq 0$:
\begin{equation}
   \label{eq:normlalisation}
n^{(\theta)}\left[1- \expp{-\lambda \sigma}\right]=
\psi_\theta^{-1}(\lambda), \quad \lambda>0, 
\end{equation}
and 
\begin{equation}
   \label{eq:normlalisation-zeta}
n^{(\theta)} (\zeta\geq h)=n^{(\theta)} (\zeta> h)=c_\theta(h), \quad h>0.
\end{equation}
\end{rem}

For  $\theta\in   \R$,  let  $\P_\theta^\uparrow(de)$  be   the  law  of
$B^{(\theta)} - 2 I^{(\theta)}$. According to \cite{b:lp} Proposition 14
and Theorem  20 in Section  VII, $\P_\theta^\uparrow(de)$ is the  law of
$B^{(\theta)}$ conditionally on being positive.  For $\theta\in \R$, let
$n_{\theta}$  be  the  excursion  measure of  $B^{(\theta)}$  outside  0
associated with the local time $L^0=L^0(B^{(\theta)})$.  For completeness,
we  give at  the end  of this  section a  proof of  the following  known
result.  Let $\cc([0,  +\infty ))$ be the set  of real-valued continuous
function defined on $[0, +\infty )$. Recall that, according to Definition \eqref{eq:def-Btheta} of $B^{(\theta)}$, the case $\theta<0$ corresponds to a positive drift for the Brownian motion.
\begin{lem}
   \label{lem:mb-drift}
We have for $\theta\in \R$ and $A\in \cc([0,
+\infty ))$ a measurable sub-set:
\begin{equation}
   \label{eq:nq=}
n_{\theta}(e\in A)=\frac{\beta}{2} \left[n^{(|\theta|)} (e\in A) +
  n^{(|\theta|)} (-e\in A)
+ 2|\theta|\P_\theta^\uparrow (-\sgn(\theta) e \in A)\right].
\end{equation}
We also have that:
\begin{equation}
   \label{eq:Pq-nq}
\P_\theta^\uparrow(de)=\P_{-\theta}^\uparrow(de)
\quad \text{and}\quad  
n^{(\theta)}(de ) \ind_{\{\sigma<+\infty \}}= n^{(|\theta|)}(de),
\end{equation}
and for $\theta<0$:
\begin{equation}
   \label{eq:nq<0}
n^{(\theta)}(\sigma=+\infty )=2|\theta| 
\quad\text{and}\quad 
n^{(\theta)}(de)\ind_{\{\sigma=+\infty\}} =
2|\theta|\P_\theta^\uparrow(de).
\end{equation}
Furthermore, if $\theta<0$, then $-\beta I^{(\theta)}_\infty $ is 
exponentially distributed with parameter  $2|\theta|$. 
\end{lem}

\begin{rem}
   \label{rem:ADnq}
   The  excursion measure  $n^{(\theta)}$ corresponds also to the
   excursion measure $\bar n^{(\theta)}$ introduced in  \cite{ad:ctvmp} of  the height
   process in the  super-critical  case, that  is for 
   $\theta<0$. Indeed  Corollary  4.4 in
   \cite{ad:ctvmp} gives that $\bar n^{(\theta)}(de )
   \ind_{\{\sigma<+\infty \}}= n^{(|\theta|)}(de)$ and 
 Lemma 4.6 in
   \cite{ad:ctvmp} gives that $\bar n^{(\theta)}(\sigma=+\infty  )=
   2|\theta|$. 
\end{rem}

Let $\theta\geq 0$ and  $\cn(dh, d \varepsilon, de)=\sum_{i\in \ci} \delta(h_i, e_i) (dh,
de)$ be a Poisson point measure on $\R_+ \times \cc([0, +\infty ))$ 
with intensity $\beta\ind_{\{h\geq 0\}}\,  dh \, n^{(\theta)}(de)$. 
For every $i\in \ci$, we set:
\[
a_i=\sum_{j\in \ci}\ind_{\{h_j<h_i\}}\sigma(e_j)\quad\mbox{and}\quad
b_i=a_i+\sigma (e_i),
\]
where  $\sigma(e_i)$  is  the  length of  excursion  $e_i$.   For  every
$t\ge   0$,  we   set  $i_t$   the  only   index  $i\in\ci$   such  that
$a_i\le  t<b_i$.   Notice that  $i_t$  is  a.s.~well  defined but  on  a
Lebesgue-null  set  of   values  of  $t$.   We   define  the  process
$(Y,J)=((Y_t, J_t),           {t\ge          0})$ by:
\[
Y_t=e_{i_t}(t-a_{i_t})
\quad\mbox{and}\quad 
J_t=h_{i_t}
\quad \text{for } t\geq 0,
\]
with the convention $Y_t=0$ and $J_t=\sup\{J_s, s<t\}$ for $t$ such that
$i_t$ is not well defined. 
Since $n^{(\theta)}$ is the excursion measure of $B^{(\theta)}
- I^{(\theta)}$ above 0 associated  to its local time at 0 given by
  $-\beta I^{(\theta)}$, we deduce the following corollary from
  excursion theory. 
\begin{cor}
   \label{cor:b+-b-}
Let $\theta\geq 0$. We have that $(Y,J)$ is distributed as $(B^{(\theta)}
- I^{(\theta)}, - I^{(\theta)})$, and thus  $Y-J$ and  $Y+J$ are respectively
distributed as  $B^{(\theta)}$ and $B^{(\theta)} - 2
 I^{(\theta)}$. 
\end{cor}

According to Theorem 1 from \cite{rp:mf} and taking into account the
scale $\sqrt{2/\beta}$, the process $B^{(\theta)} - 2
 I^{(\theta)}$ is a diffusion on $[0, +\infty )$ with infinitesimal
 generator:
\begin{equation}
   \label{eq:generator}
\beta^{-1} \,  \partial^2_x + 2|\theta| \coth (\beta |\theta|
x)\,  \partial_x, \quad x\in [0, +\infty ).
\end{equation}

\begin{proof}[Proof of Lemma \ref{lem:mb-drift}]
  Since $B^{(\theta)} - 2
I^{(\theta)}$ and $B^{(-\theta)} - 2 I^{(-\theta)}$ have the same
distribution, see Theorem 1 from \cite{rp:mf}, we deduce that
$\P_\theta^\uparrow(de)=\P_{-\theta}^\uparrow(de)$, which gives the
first part of \reff{eq:Pq-nq}. 
\medskip 

For $\theta, \lambda\in \R $, we  set $\varphi_\theta(\lambda)
=\psi_\theta(\lambda/\beta)=\beta ^{-1} \lambda^2
- 2\theta \lambda$ 
so that $\E[\exp( \lambda B^{(\theta)}_t)]=\exp(t
  \varphi_\theta(\lambda))$. Elementary computations gives:
\[
\int_0^\infty  \expp{-\lambda x} c_\theta(x)^{-1}  \,
dx=\inv{\varphi_\theta(\lambda)}\quad\text{for all $\lambda>2\beta \max(\theta, 0)$}.
\]
This implies that $1/c_\theta$ is the scale function of $ B^{(\theta)}$,
see Theorem 8  in Section VII from \cite{b:lp}. Thanks  to Theorem 8 and
Proposition 15  in Section  VII from  \cite{b:lp}, there exists a
positive constant $k_\theta$ such that   for all
$t>0$   and   $A$   in   the   $\sigma$-field   $\ce_t$   generated   by
$(e(s), s\leq t)$:
\begin{equation}
   \label{eq:kE}
n^{(\theta)} (A, \sigma>t) = k_\theta \E_\theta^\uparrow \left[
  c_\theta( e(t)) \ind_A \right],
\end{equation}
 and $ n^{(\theta)} (\zeta>h) = k_\theta c_\theta(h)$ for all $h>0$.
We deduce from \reff{eq:normlalisation-zeta} and the latter equality that
$k_\theta=1$ for $\theta\geq 0$.

We now prove that $k_\theta=1$ also for $\theta<0$. Assume that
$\theta<0$. Letting $t$ goes to infinity in \reff{eq:kE} (with $A$
fixed) and using  that
$\P_\theta^\uparrow(de)$-a.s. $\lim_{t\rightarrow +\infty } e(t)=+\infty
$, we deduce that:
\begin{equation}
   \label{eq:nq-Pq-k}
n^{(\theta)}( A, \sigma=+\infty )=2|\theta| \, k_\theta
\P_\theta^\uparrow(A) \quad\text{for all $A\in \ce_t$ and all $t\geq 0$,}
\end{equation} 
and taking for $A$ the whole state space, we get that:
\begin{equation}
   \label{eq:nq=infty}
n^{(\theta)}(\sigma=+\infty
)=2|\theta| k_\theta. 
 \end{equation} 
 By  the excursion  theory and  the  chosen normalization,  we get  that
 $  -  \beta  I^{(\theta)}_\infty  $  is  exponential  with  parameter
 $n^{(\theta)}(\sigma=+\infty)$.  Since  by scaling  $ -I^{(\theta)}_\infty$
 is also distributed as $\inf\{ B_t -  \beta \theta t, \, t\geq 0\}$, we
 deduce  from IV-5-32  p.~70 in  \cite{bs:hbm}, that  $ -I^{(\theta)}_\infty$  is
 exponential   with   parameter   $2\beta  |\theta|$   and   thus   that
 $   -\beta   I^{(\theta)}_\infty$   is   exponential   with   parameter
 $2 |\theta|$. This implies that
 $n^{(\theta)}(\sigma=+\infty)=2|\theta|$, which gives the first part of \reff{eq:nq<0}
 and thus, thanks to \reff{eq:nq=infty}, we get $k_\theta=1$. Then use
 \reff{eq:nq-Pq-k} with  $k_\theta=1$ to get the second part of
 \reff{eq:nq<0}. 
\medskip

Let $\theta<0$. Let $t>0$ and 
$A\in \ce_t$. 
We have:
\begin{align*}
n^{(\theta)} (A, \sigma>t)
&= \E_\theta^\uparrow \left[
  c_\theta( e(t)) \ind_A \right]\\
&= 2|\theta| \P_\theta^\uparrow(A) + \E_{|\theta|}^\uparrow
  \left[c_{|\theta|}(h)( e(t)) \ind_A 
\right]\\
& = n^{(\theta)} (A, \sigma=+\infty ) + n^{(|\theta|)} (A, \sigma>t),
\end{align*}
where we used 
\reff{eq:kE} and $k_\theta=1$ for the first  equality; that  
$c_{\theta}(h)=2|\theta| + c_{|\theta|}(h)$ for all $h>0$, thanks to 
\reff{eq:def-c} and that
$\P_\theta^\uparrow(de)=\P_{-\theta}^\uparrow(de)$ for the second; and 
\reff{eq:kE} with $|\theta|$ instead of $\theta$ and $k_{|\theta|}=1$ as
well as \reff{eq:nq-Pq-k}
for the last. This implies the last part of \reff{eq:Pq-nq} for
$\theta<0$. We also deduce from
\reff{eq:normlalisation} that $n^{(\theta)} (\sigma=+\infty )=0$  for
$\theta\geq 0$. Thus the second part of \reff{eq:Pq-nq} holds also for
$\theta\geq 0$ and thus for $\theta\in \R$. 
\medskip

We shall  now prove  \reff{eq:nq=}. Recall  $n_\theta$ is  the excursion
measure  of $B^{(\theta)}$  outside  $0$ associated with the local  time
$L^0=L^0(B^{(\theta)})$, $n^{(\theta)}(de)$ is  the excursion measure of
$B^{(\theta)}-  I^{(\theta)}$ above  0, and  $n^{(-\theta)}(de)$ is  the
excursion measure  of $B  ^{(-\theta)}- I^{(-\theta)}$ above  0.  Notice
that    $B    ^{(-\theta)}-    I^{(-\theta)}$    is    distributed    as
$-(            B^{(\theta)}-            M^{(\theta)})$,            where
$M^{(\theta)}     =(M_t^{(\theta)},     t\geq    0)$,     defined     by
$M^{(\theta)}_t=\sup_{u\in  [0,  t]}  B^{(\theta)}_u$,  is  the  maximum
process,  and thus  $n^{(-\theta)}(d(-e))$ is  the excursion  measure of
$B  ^{(\theta)}- M^{(\theta)}$  below  0.   According to  \cite{b:dmdml}
p.~334 (which  is stated  for $\beta=2$  but can  clearly be  stated for
$\beta>0$ using a scaling  in time), we get that:  i) $n^{(\theta)}(de)$, the
excursion measure of $B^{(\theta)}- I^{(\theta)}$  above 0, is equal, up
to a multiplicative  constant due to the choice of  the normalization of
the    local    times,    to    $\ind_{\{e>0\}}    n_\theta(de)$;    ii)
$n^{(-\theta)}(d(-e))$,       the       excursion       measure       of
$B ^{(\theta)}- M^{(\theta)}$ below 0,  is equal, up to a multiplicative
constant due to  the choice of the normalization of  the local times, to
$\ind_{\{e>0\}} n_\theta(de)$. Thus, we have, for some positive constant
$a_\theta$ and $b_\theta$, that:
\[
n_\theta(de)=a_\theta n^{(\theta)}(de) + b_\theta n^{(-\theta)}(d(-e)).
\]
Thanks to \reff{eq:Pq-nq} and \reff{eq:nq<0}, we get that \reff{eq:nq=}
is proved once we prove that $a_\theta=b_\theta=\beta/2$. 
\medskip

Let  us assume  for  simplicity  that $\theta\geq  0$  (the argument  is
similar for $\theta\leq  0$). By the excursion theory,  $L^0_\infty $ is
exponential                        with                        parameter
$n_\theta(\sigma=+\infty  )=  b_\theta n^{(-\theta)}(\sigma=+\infty  )=2
\theta                                                        b_\theta$.
According to V-3-11 p.~90 in \cite{bs:hbm}, $L^0_\infty $ is exponential
with  parameter  $\beta   \theta$  (use  that  $(L^0_t,   t\geq  0)$  is
distributed as  $(L^0_{2t/\beta} (W), t\geq 0)$  the local time at  0 of
the Brownian motion $W=(W_t=B_t -  \beta\theta t, t\geq 0)$). This gives
$b_\theta=\beta/2$.

We now prove that $a_\theta=\beta/2$. Let $T= \sup\{B^{(\theta)}_t, t\geq 0\}$. We have:
\[
\P(T<a)
=\E\left[\expp{- n_{\theta}(\zeta\geq a, e>0)\, L^0_\infty  }\right]
=\E\left[\expp{- a_\theta\, n^{(\theta)} (\zeta\geq a)\,L^0_\infty }\right]
= \E\left[\expp{- a_\theta c_\theta(a) \, L^0_\infty }\right]
= \frac{\beta\theta }{\beta\theta+ a_\theta c_\theta(a)},
\]
where we used that $L^0_\infty $  is exponential with
parameter $\beta \theta$ for the last equality.
  Since  by scaling  $ T$ 
 is also distributed as $\sup\{ B_t -  \beta \theta t, \, t\geq 0\}$, we
 deduce  from IV-5-32  p.~70 in  \cite{bs:hbm}, that  $ T$  is
 exponential   with   parameter   $2\beta  \theta$. 
This gives $\P(T<a)=1 - \expp{- 2\beta\theta a}$. Using \reff{eq:def-c},
we deduce that $a_\theta=\frac{\beta}{2}$. This ends the proof of the
lemma. 
\end{proof}

\subsection{Real trees}
\label{sec:real}

The study of real trees has been motivated by algebraic and geometric purposes.
See in particular the survey \cite{dmt:tto}. It has been first used in
\cite{epw:rprtrgr} to study random continuum trees, see also
\cite{e:prt}.

\begin{defi}[Real tree]
\label{defi:realtree}
A real tree is a metric space $(\bt,d_\bt)$ such that:
\begin{itemize}
\item[(i)] For every $x,y\in\bt$, there is a unique isometric map $f_{x,y}$
  from $[0,d_\bt(x,y)]$ to $\bt$ such that $f_{x,y}(0)=x$ and $f_{x,y}(d_\bt(x,y))=y$.
\item[(ii)] For every $x,y\in\bt$, if $\phi$ is a continuous injective map from $[0,1]$ to $\bt$ such that $\phi(0) = x$ and
$\phi(1) = y$, then $\phi([0, 1]) = f_{x,y}([0, d_\bt(x,y)])$.
\end{itemize}
\end{defi}

Notice that a real tree is  a length space as defined in \cite{bbi:cmg}.
We say that  $(\bt, d_\bt, \partial_\bt)$ is a {\sl  rooted } real tree,
where $\partial=\partial_\bt$ is a  distinguished vertex of $\bt$, which
will be called the root.  Remark that the set $\{\partial\}$ is a rooted
tree that only contains the root.

Let $\bt$ be a compact rooted real tree  and let  $x,y\in\bt$. We denote by $\lb x,y\rb$
the   range    of   the   map   $f_{x,y}$    described   in   Definition
\ref{defi:realtree}. We also  set $\lb x,y\lb=\lb x,y\rb\setminus\{y\}$.
We define the out-degree of $x$, denoted by $k_\bt(x)$, as the number of
connected  components of  $\bt\setminus\{x\}$  that do  not contain  the
root. If  $k_\bt(x)=0$, resp. $k_\bt(x)>1$,  then $x$ is called  a leaf,
resp. a branching point.  A tree is said to be
binary if  the out-degree of  its vertices belongs to  $\{0,1,2\}$.  The
skeleton of the  tree $\bt$ is the  set $\mathrm{sk}(\bt)$ of points of $\bt$  that are not
leaves.  Notice     that
$\clo (  \mathrm{sk}(\bt))=\bt$, where  $\clo(A)$ denote the  closure of
$A$.

We denote by $\bt_x$ the
sub-tree of $\bt$ above $x$ \textit{i.e.}~$$\bt_x=\{y\in\bt,\ x\in \lb\partial,y\rb\}$$
rooted at $x$.
We say that $x$ is an ancestor of $y$, which we denote
by $x\preccurlyeq y$, if $y\in \bt_x$. We write 
$x\prec y$ if furthermore $x\neq y$. Notice that $\preccurlyeq$ is a
partial order on $\bt$. 
We denote by $x\wedge y$ the Most Recent Common Ancestor (MRCA) of $x$ and $y$ in
$\bt$ \textit{i.e.}~the unique vertex of $\bt$ such that
$\lb\partial,x\rb\cap\lb\partial,y\rb=\lb\partial, x\wedge y\rb$.

We denote by
$h_\bt(x)=d_\bt(\partial,x)$ the height of the vertex $x$ in the tree
$\bt$ and by $H(\bt)$ the height of the tree $\bt$:
$$H(\bt)=\max\{h_\bt(x),\ x\in\bt\}.$$

% For $\varepsilon >0$, we define the erased tree $r_\varepsilon(\bt)$
% (sometimes called in the literature the $\varepsilon$-trimming of the tree $\bt$)
% by
% \[
%   r_\varepsilon(\bt)=\{x\in\bt\backslash\{\partial\},\ H(\bt_x)\ge \varepsilon\}\cup
% \{\partial\}.
% \]
% For $\varepsilon>0$, $r_\varepsilon(\bt)$ is indeed a tree
% and $r_\varepsilon(\bt)=\{\partial\}$ for $\varepsilon> H(\bt)$. Notice that
% \begin{equation}%\label{eq:approximation}
% \bigcup_{\varepsilon>0}r_\varepsilon(\bt)=\mathrm{sk}(\bt).
% \end{equation}

Recall  $\bt$ is  a compact rooted real tree and let $(\bt_i, {i\in I})$ be a
family of rooted trees, and $(x_i, {i\in  I})$ a family of  vertices of $\bt$.
We   denote  by   $\bt_i^\circ=\bt_i\setminus\{\partial_{\bt_i}\}$.   We
define  the  tree  $\bt\circledast_{i\in  I}  (\bt_i,x_i)$  obtained  by
grafting the trees $\bt_i$ on the tree $\bt$ at points $x_i$ by
\begin{align*}
 & \bt\circledast_{i\in I}(\bt_i,x_i)=\bt\sqcup\left(\bigsqcup_{i\in I} \bt_i^\circ\right),\\
&
d_{\bt\circledast_{i\in I}
  (\bt_i,x_i)}(y,y') =\begin{cases}
d_\bt(y,y') & \mbox{if }y,y'\in \bt,\\
d_{\bt_i}(y,y') & \mbox{if }y,y'\in \bt_i^\circ,\\
d_\bt(y,x_i)+d_{\bt_i}(\partial_{\bt_i},y') & \mbox{if } y\in\bt\mbox{
  and }y'\in\bt_i^\circ,\\
d_{\bt_i}(y,\partial_{\bt_i})+d_\bt(x_i,x_j)+d_{\bt_j}(\partial_{\bt_j},y')
&  \mbox{if } y\in\bt_i^\circ\mbox{
  and }y'\in\bt_j^\circ\mbox{ with }i\ne j,
\end{cases}\\
 & \partial_{\bt\circledast_{i\in I}(\bt_i,x_i)}=\partial_\bt,
\end{align*}
where $A\sqcup B$ denotes the disjoint union of the sets $A$ and
$B$. Notice that $\bt\circledast_{i\in I}(\bt_i,x_i)$ might not be
compact. 
\medskip

We say that  two rooted real trees $\bt$ and  $\bt'$ are equivalent (and
we note $\bt\sim\bt'$)  if there exists a  root-preserving isometry that
maps $\bt$ onto $\bt'$. We denote by $\T$ the set of equivalence classes
of compact rooted real trees.   The metric space $(\T,d_{GH})$, with the
so-called Gromov-Hausdorff  distance  $d_{GH}$,  is Polish,  see
\cite{epw:rprtrgr}. This allows to define random real trees.

\subsection{Coding a compact real tree by a function and the Brownian
  CRT}\label{sec:CRT}

Let     $\ce$     be     the      set     of     continuous     function
$g:[0,+\infty)\longrightarrow [0,+\infty)$ with compact support and such
that     $g(0)     =     0$.      For     $g\in     \ce$,     we     set
$\sigma(g)=\sup  \{x, \,  g(x)>0\}$.  Let $g\in  \ce$,  and assume  that
$\sigma(g)>0$,  that  is   $g$  is  not  identically   zero.  For  every
$s, t \ge 0$, we set:
\[
m_g(s, t) = \inf _{r\in [s\wedge t,s\vee  t ]}g(r),
\]
and
\begin{equation}\label{eq:dist_g}
d_g(s, t) = g(s) + g(t) - 2m_g(s, t ).
\end{equation}
It is easy to check that  $d_g$ is a pseudo-metric on $[0,+\infty)$.  We
then say that $s$ and $t$ are equivalent  iff $d_g(s, t) = 0$ and we set
$T_g$ the associated quotient space. We  keep the notation $d_g$ for the
induced distance  on $T_g$.   Then the  metric space  $(T_g, d_g)$  is a
compact  real-tree, see  \cite{dlg:pfalt}.   We denote  by  $p_g$  the
canonical  projection  from $[0,  +\infty  )$  to  $T_g$. We  will  view
$(T_g, d_g)$  as a rooted real  tree with root $\partial  = p_g(0)$. We
will call $(T_g,d_g)$ the  real tree coded by $g$, and conversely that $g$ is a contour function of the tree $T_g$. We  denote by $F$ the
application that associates with a function $g\in\ce$ the equivalence class of the tree $T_g$.

Conversely every rooted compact real tree $(T,d)$ can be coded by a
continuous function $g$ (up to a root-preserving isometry), see
\cite{d:ccrtrvf}.
\medskip

We define the  Brownian CRT, $\tau=F (e)$, as the  (equivalence class of
the) tree coded by the  positive excursion $e$ under $n^{(\theta)}$, see
Section \ref{sec:em-bmd}.  And we define the  measure $\N^{(\theta)}$ on
$\T$ as the ``distribution'' of $\tau$,  that is the push-forward of the
measure   $n^{(\theta)}$   by   the   application   $F$.   Notice   that
$H(\tau)=\zeta(e)$.

Let     $e$   be  with     ``distribution''     $n^{(\theta)}(de)$     and
let $(\Lambda_s^a, s\ge 0,  a\ge 0)$ be the local  time of $e$ at  time $s$ and
level $a$.  Then, we  define the  local time measure  of $\tau$  at level
$a\ge 0$,  denoted by $\ell_a(dx)$,  as the push-forward of  the measure
$d\Lambda_s^a$   by  the   map  $  F$,  see   Theorem  4.2   in
\cite{dlg:pfalt}. We shall define $\ell_a$ for $a\in \R$ by setting
$\ell_a=0$ for $a\in \R\setminus [0, H(\tau)]$.

\subsection{Trees with one semi-infinite branch}
\label{sec:forest}

The  goal of  this section  is to  describe the  genealogical tree  of a
stationary  CB  with  immigration  (restricted to  the  population  that
appeared before time 0). For this purpose, we add an immortal individual
living from  $-\infty$ to 0 that  will be the spine  of the genealogical
tree (\textit{i.e.}~the semi-infinite branch) and will be represented by the half
straight line  $(-\infty,0]$, see Figure  \ref{fig:infinite_tree}. Since
we are  interested in  the genealogical  tree, we don't record  the population
generated by  the immortal  individual after  time 0.  The distinguished
vertex in the  tree will be the point  0 and hence would be  the root of
the tree in  the terminology of Section \ref{sec:real}.  We will however
speak of the distinguished  leaf in what follows in order  to fit with the
natural intuition. In  the same spirit, we will  give another definition
for the  height of a vertex  in such a  tree in order to  allow negative
heights.

\begin{figure}[H]
\includegraphics[width=7cm]{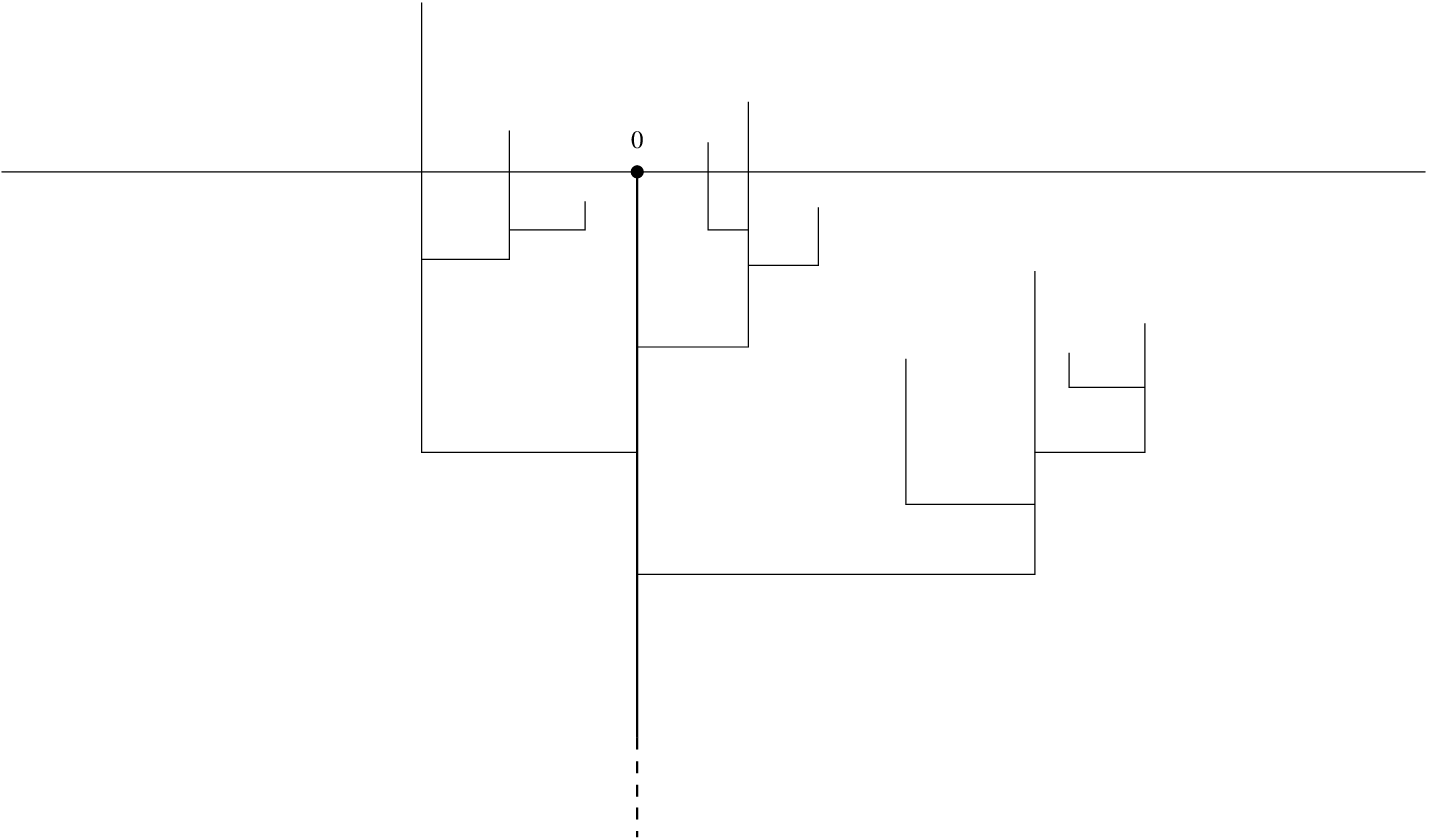}
\caption{An instance of a tree with a semi-infinite branch}\label{fig:infinite_tree}
\end{figure}

\subsubsection{Forests}
\label{sec:forests}
A  forest  $\bff$  is  a  family $((h_i,\bt_i), \, i\in  I)$  of  points  of
$\R\times \T$. Using  an immediate extension of  the grafting procedure,
for  an  interval $\mathfrak{I}\subset  \R$,  we  define the  real  tree
\begin{equation}\label{eq:tree-forest}
\bff_\mathfrak{I}=\mathfrak{I}\circledast_{i\in I, h_i\in \mathfrak{I}}
(\bt_i,                                                           h_i).
\end{equation}

Let us denote, for $i\in I$,  by $d_i$ the distance of the tree $\bt_i$
and by $\bt_i^\circ=\bt_i\setminus\{\partial_{\bt_i}\}$ the tree $\bt_i$
without its root. The
distance on $\bff_\mathfrak{I}$ is then defined, for $x,y\in
\bff_\mathfrak{I}$, by: 
\[
d_\bff(x,y)=\begin{cases}
d_i(x,y) & \text{ if }x,y\in\bt_i^\circ,\\
h_{\bt_i}(x)+|h_i-h_j|+h_{\bt_j}(y) & \text{ if } x\in\bt_i^\circ,\ y\in\bt_j^\circ\mbox{ with }i\ne j,\\
|x-h_j|+h_{\bt_j}(y) & \text{ if } x\not\in\bigsqcup_{i\in I}\bt_i^\circ,\
y\in\bt_j^\circ\\ 
|x-y| & \text{ if } x,y \not\in\bigsqcup_{i\in I}\bt_i^\circ.
\end{cases}
\]
 
Let us recall the following lemma (see \cite{ad:rpbt}).

\begin{lem}\label{lem:loc-comp}
  Let  $\mathfrak{I}\subset  \R$ be  a  closed  interval. If  for  every
  $a,b\in\mathfrak{I}$, such that $a<b$,  and every $\varepsilon>0$, the
  set $\{i\in I,\ h_i\in[a,b],\  H(\bt_i)>\varepsilon\}$ is finite, then
  the tree $\bff_\mathfrak{I}$ is a complete locally compact length
  space. 
\end{lem}

\subsubsection{Trees with one semi-infinite branch}
\label{sec:inf-branch}

\begin{defi}\label{def:T_1}
We set $\T_1$ the set of forests $\bff=((h_i,\bt_i), \, i\in I)$ such that
\begin{itemize}
\item for every $i\in I$, $h_i\le 0$,
\item for every
   $a<b$,       and      every       $\varepsilon>0$,      the       set
   $\{i\in  I,\ h_i\in[a,b],\  H(\bt_i)>\varepsilon\}$ is  finite.
\end{itemize}
\end{defi}

The following corollary, which is an elementary consequence of Lemma \ref{lem:loc-comp}, associates with a forest $\bff\in\T_1$ a complete and locally compact real tree.

\begin{cor}
   \label{cor:bff}
   Let $\bff=((h_i,\bt_i), \,  i\in I)\in \T_1$. Then,
   the  tree $\bff_{(-\infty,0]}$ defined by \reff{eq:tree-forest} is a complete and locally compact real tree.
\end{cor}

Conversely, let  $(\bt,d_\bt,  \rho_0)$ be  a complete and locally compact rooted  real  tree.  We denote by $\cs(\bt)$ the set of vertices $x\in\bt$ such that at
least one of the connected components of $\bt\setminus\{x\}$ that do not
contain $\rho_0$ is unbounded.  If $\cs(\bt)$ is not empty, then it is a
tree  which  contains  $\rho_0$.   We   say  that  $\bt$  has  a  unique
semi-infinite branch  if $\cs(\bt)$  is non-empty  and has  no branching
point. We   set  $(\bt_i^\circ,i\in  I)$  the   connected  components  of
$\bt\setminus\cs(\bt)$.  For every  $i\in I$,  we set  $x_i$ the  unique
point  of $\cs(\bt)$  such that  $\inf\{d_\bt(x_i,y),\ y\in\bt_i^\circ\}=0$,
and:
\[
\bt_i=\bt_i^\circ\cup\{x_i\},\qquad h_i=-d(\rho_0,x_i).
\]
We  shall say  that $x_i$  is the  root of  $\bt_i$.  Notice  first that
$(\bt_i,  d_\bt, x_i)$  is a  bounded rooted  tree. It  is also  compact
since,  according  to the  Hopf-Rinow  theorem  (see Theorem  2.5.26  in
\cite{bbi:cmg}), it  is a  bounded closed subset  of a  complete locally
compact length space.  Thus it belongs to $\T$.

The family $\bff=((h_i,\bt_i), \, i\in I)$ is therefore a forest with $h_i<0$. To check that it belongs to $\T_1$, we need to prove that the second condition in Definition \ref{def:T_1} is satisfied which is a direct consequence of the fact that the tree $\bff_{[a,b]}$ is locally compact.\medskip

We  can therefore identify the set $\T_1$ with the set of (equivalence classes) of complete locally compact rooted real trees with a unique semi-infinite branch.      We can follow  \cite{adh:nghpdblcmms} to endow $\T_1$ with
a Gromov-Hausdorff-type distance for which $\T_1$ is a Polish
space.\medskip

We extend  the partial  order defined  for trees in  $\T$ to  forests in
$\T_1$, with the  idea that the distinguished leaf $\rho_0=0$  is at the
tip of the semi-infinite branch.  Let $\bff=(h_i,\bt_i)_{i\in I}\in\T_1$
and  write $\bt=\bff_{(-\infty,0]}$  viewed  as a  real  tree rooted  at
$\rho_0=0$ (with a  unique semi-infinite branch $\cs(\bt)=(-\infty,0]$).
For $x,y\in \bt$, we set $x\preccurlyeq y$ if either $x,y\in S(\bt)$ and
$d_\bff(x,\rho_0) \geq  d_\bff(y, \rho_0)$,  or $x,y\in \bt_i$  for some
$i\in I$  and $x\preccurlyeq y$ (with  the partial order for  the rooted
compact real tree $\bt_i$), or $x\in \cs(\bt)$ and $y\in \bt_i$ for some
$i\in  I$  and $d_\bff(x,\rho_0)\ge  |h_i|$.   We  write $x\prec  y$  if
furthermore $x\neq y$.   We define $x\wedge y$ the MRCA  of $x,y\in \bt$
as $x$ if $x\preccurlyeq y$, as $x  \wedge y$ if $x,y\in \bt_i$ for some
$i\in I$  (with the MRCA for  the rooted compact real  tree $\bt_i$), as
$h_i\wedge h_j$ if $x\in \bt_i$ and $y\in \bt _j$ for some $i\neq j$. We
define the height of a vertex $x\in \bt$ as
\[
h_\bff(x)= d_\bff(x,  \rho_0 \wedge
x)-  d_\bff(\rho_0, \rho_0  \wedge x).
\]
Notice  that   the      definition     of  the   height function $h_\bff$    for      a     forest
$\bff=(h_i,\bt_i)_{i\in I}\in \T_1$ is  different than the height function of the
tree $\bt=\bff_{(-\infty,0]}$ viewed as a tree in $\T$,  as in the former  case the root $\rho_0$ is  viewed as a
distinguished vertex above  the semi-infinite branch (all elements of this semi-infinite branch have negative heights for $h_\bff$ whereas all the heights are nonnegative for $h_\bt$).

\subsubsection{Coding a forest by a contour function}
Construction of a tree of the type $\bff_{[0, +\infty)}$ via a contour
function as in Section \ref{sec:CRT} is already present  in
\cite{d:crti} and in  Section
7.4 from  \cite{alw:ip}. This construction is recalled in section
\ref{sec:local_times}. We now present a   construction   of  a   tree  of   the  type
$\bff_{(-\infty,0]}$ via a contour function as in Section \ref{sec:CRT}.
Let $\ce_1$ be the set of  continuous functions $g$ defined on $\R$ such
that $g(0)=0$ and
$\liminf_{x\rightarrow -\infty  } g(x)= \liminf_{x\rightarrow  +\infty }
g(x)=-\infty$.
For such a  function, we still consider the  pseudo-metric $d_g$ defined
by \reff{eq:dist_g}  (but for $s,t\in\R$)  and define the tree  $T^-_g$ as
the quotient space on $\R$ induced  by this pseudo-metric.  We set $p_g$
as the canonical projection from $\R$ onto $T^-_g$.

\begin{lem}\label{lem:TginT1}
  Let $g\in  \ce_1$. The  triplet $(T^-_g,d_g,p_g(0))$  is a complete locally compact rooted real tree with a unique semi-infinite branch.
\end{lem}

When there is no possible confusion, we write $T_g$ for $T^-_g$. 

\begin{proof}
  We  define the  infimum  function  $\underline g(x)$  on  $\R$ as  the
  infimum       of       $g$       between      $0$       and       $x$:
  $\underline  g(x)= \inf_{[x  \wedge 0,  x \vee  0]} g$.   The function
  $g-\underline    g$    is    non-negative    and    continuous.    Let
  $((a_i, b_i), i\in I)$ be  the excursion intervals of $g-\underline g$
  above  $0$.   Because  of the   hypothesis  on  $g$,   the  intervals
  $(a_i,  b_i)$  are   bounded.  For  $i\in  I$,   set  $h_i=g(a_i)$  and
  $g_i(x)=g((a_i+x)\wedge b_i) - h_i$ so that $g_i\in \ce$. Consider the
  forest $\bff=((h_i,T_{g_i}), \, i\in I) $.

  It         is         elementary         to         check         that
  $(\bff_{(-\infty ,  g(0)]}, d_\bff, g(0))$ and  $(T_g,d_g,p_g(0))$ are
  root-preserving and isometric.  To conclude,  it is  enough to  check that
  $\bff\in\T_1$. First remark that, by  definition, $h_i\le 0$ for every
  $i\in         I$.           Let         $r>0$          and         set
  $r_\text{g}=\inf\{x,   \,   \underline   g   (x)\geq   g(0)-r\}$   and
  $r_\text{d}=\sup\{x, \, \underline g (x)\geq g(0)-r\}$. Because of the
  hypothesis  on $g$,  we have  that $r_\text{g}$  and $r_\text{d}$  are
  finite.     By     continuity    of     $g-    \underline     g$    on
  $[r_\text{g}, r_\text{d}]$,  we deduce  that for  any $\varepsilon>0$,
  the                                                                set
  $\{i\in I; \, (a_i, b_i)\subset  [r_\text{g}, r_\text{d}] \text{ and }
  \sup_{(a_i,     b_i)}     (g-\underline    g)     >\varepsilon     \}$
  is   finite.    Since   this   holds    for   any   $r>0$   and   that
  $H(T_{g_i})=\sup_{(a_i, b_i)} (g-\underline g) $  for all $i\in I$, we
  deduce that $\bff\in\T_1$. This concludes the proof.
\end{proof}

\subsubsection{Genealogical tree of an extant population}
\label{sec:genealogy}
For a tree $\bt\in \T$ or $\bt\in \T_1$ (recall that we identify a forest $\bff\in\T_1$ with the tree $\bt=\bff_{(-\infty,0]}$ with a different definition for the height function) and $h\geq 0$, we define $\cz_h(\bt)=\{x\in \bt, \,
h_\bt(x)=h\}$ the set of vertices of $\bt$ at level $h$ also called the
extant population at time $h$, and 
the genealogical tree of the vertices
of $\bt$ at level $h$ by:
\begin{equation}
   \label{eq:def-G}
\cg_h(\bt)=\{x\in \bt; \, \exists y\in \cz_h(\bt) \text{ such that } x
\preccurlyeq y\}.
\end{equation}
For $\bff\in\T_1$, we write $\cg_h(\bff)$ for
$\cg_h(\bff_{(-\infty,0]})$; 

%For $\bt\in \T$ and $h\in [0, H(\bt)]$, $\cg_h(\bt)$ is indeed a tree and
%$\cg_h(\bt)=\emptyset$ for $h>H(\bt)$. For $(\bt, d_\bt, \rho_0, h_0)\in
%\ct_1$, $\cg_h(\bt)$ is a tree at least if $h\leq h_0$. 
%
%For  a  forest $\bff$,  we  write  $\cz_h(\bff)$ and  $\cg_h(\bff)$  for
%$\cz_h(\bff_{(-\infty , h_0]})$ and  $\cg_h(\bff_{(-\infty , h_0]})$ for
%any  $h_0\geq  h$.   Notice  that  for $h$  given,  the  definitions  of
%$\cz_h(\bff)$ and $\cg_h(\bff)$ do not depend on $h_0\geq h$. 

\section{Ancestral process}
\label{sec:ancestral}

Usually,  the ancestral  process  records the  genealogy  of $n$  extant
individuals  at time  0 picked  at  random among  the whole  population.
Using the ideas of \cite{ap:cbpmb}, we  are able to describe in the case
of a Brownian  forest the genealogy of all extant  individuals at time 0
by a simple Poisson point process on $\R^2$.

\subsection{Construction of a tree from a point measure}\label{sec:constructionT}

\begin{defi}
\label{def:anc}
A point process $\ca(dx,d\zeta)=\sum_{i\in
  \ci}\delta_{(x_i,\zeta_i)}(dx,d\zeta)$ on $\R^*\times (0,+\infty)$ is
said to be an ancestral process if 
\begin{itemize}
\item[(i)] $\forall i,j\in \ci $, $i\ne j\Longrightarrow x_i\ne x_j$.
\item[(ii)] $\forall a,b\in\R$, $\forall \varepsilon>0$, $\ca([a,b]\times[\varepsilon,+\infty))<+\infty$.
\item[(iii)] 
  $\sup\{\zeta_i,x_i>0\}=+\infty$ if $\sup_{i\in \ci}x_i=+\infty$; and
  $\sup\{\zeta_i,x_i<0\}=+\infty$ if $\inf_{i\in \ci}x_i=-\infty$. 
\end{itemize}
\end{defi}

Let $\ca=\sum_{i\in \ci}\delta_{(x_i,\zeta_i)}$ be a point process on
$\R^*\times [0,+\infty)$ satisfying (i) and (ii) from Definition
\ref{def:anc}. We shall associate with this ancestral process a
genealogical tree. Informally the genealogical tree is constructed as
follows. We view this process as a sequence of vertical segments in
$\R^2$, the tips of the segments being the $x_i$'s and their lengths
being the $\zeta_i$'s. We then attach the bottom of each segment such
that $x_i>0$ (resp. $x_i<0$) to the first left (resp. first right)
longer segment or to the half line $\{0\}\times (-\infty,0]$ if such a
segment does not exist. This gives a (unrooted, non-compact) real tree
that may not be complete. See also Figure \ref{fig:ancestral-intro} for an example.

Let us turn to a more formal definition. 
Let us  denote by  $\ci  ^\text{d}=\{i\in\ci,\  x_i>0\}$  and
$\ci ^\text{g}=\{i\in\ci,\ x_i<0\}=  \ci\setminus\ci^\text{d}$.  We also
set $\ci_0=\ci\sqcup \{0\}$, $x_0=0$  and $\zeta_0=+\infty$. We set, for
every  $i\in\ci_0$, $S_i=\{x_i\}\times(-\zeta_i,0]$ the  vertical
segment   in    $\R^2$   that    links   the   points    $(x_i,0)$   and
$(x_i,-\zeta_i)$.  Notice that we  omit  the lowest  point  of the  vertical
segments. Finally we define
\begin{equation}
   \label{eq:defTT}
\mathfrak{T}=\bigsqcup_{i\in\ci_0}S_i.
\end{equation}
We now define a distance on  $\mathfrak{T}$. We first define the
distance between  leaves of $\mathfrak{T}$, \textit{i.e.}~points $(x_i,0)$ with
$i\in\ci_0$, then we extend it to every point of $\mathfrak{T}$. For
$i,j\in\ci_0$ such that $x_i<x_j$, we set 
\begin{equation}\label{eq:dist-anc}
d((x_i,0),(x_j,0))=2 \max\{\zeta_k, \,x_k \in J(x_i,x_j)\},
\end{equation}
where, for  $x<y$, $J(x,y)=  (x,y] $ (resp.  $[x,y)$, resp.  $[x,y]\backslash \{0\}$) if
$x\geq 0$ (resp. $y\leq 0$, resp. $x<0$ and $y> 0$), with the convention
$\max\emptyset=0$.  For $u=(x_i,a)\in  S_i$ and  $v=(x_j,b)\in S_j$,  we
set, with $r=\frac{1}{2}d((x_i, 0), (x_j, 0))$:
\begin{equation}\label{eq:dist-anc2}
d(u,v)=|a-b|\ind_{\{x_i=x_j\}}+ (|a-r|+|b-r|)\ind_{\{x_i\neq x_j\}}.
\end{equation}
See Figure \ref{fig:dist-from-A}  for an example.  It is  easy to verify
that $d$ is a distance  on $\mathfrak{T}$. Notice that $\mathfrak{T}$ is
not compact in particular because  of the infinite half-line attached to
$(0,0)$. In order to stick to the framework of Section \ref{sec:forest},
the origin  $(0,0)$ will  be the  distinguished point  in $\mathfrak{T}$
located at height $h=0$.

\begin{figure}[H]
\begin{center}
\psfrag{xi}{$x_i$}
\psfrag{xj}{$x_j$}
\psfrag{a}{$a$}
\psfrag{b}{$b$}
\psfrag{u}{$u$}
\psfrag{v}{$v$}
\psfrag{r-a}{$r-a$}
\psfrag{r-b}{$r-b$}
\includegraphics[width=13cm]{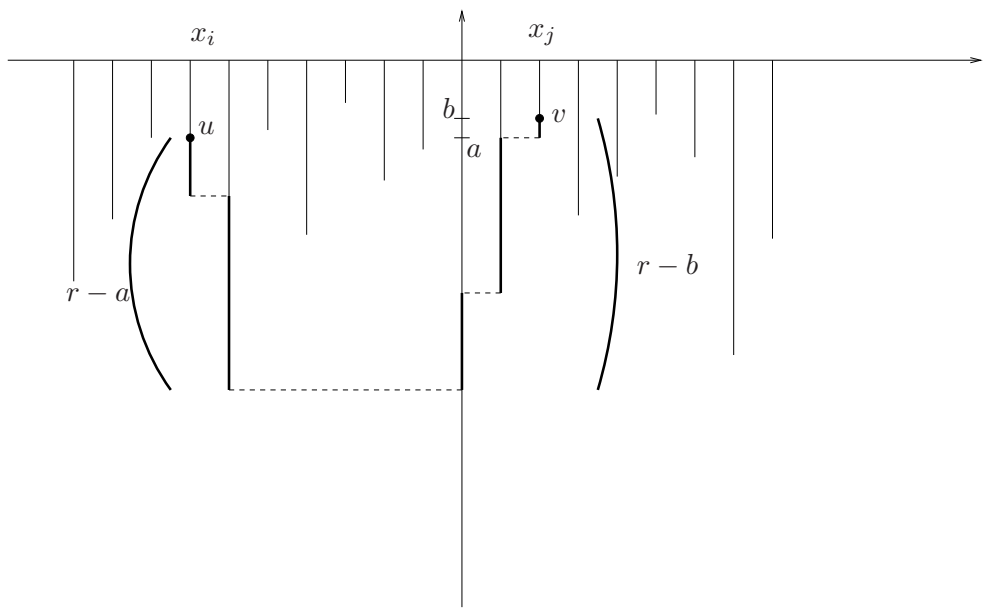}
\caption{An example of the distance $d(u,v)$ defined in \reff{eq:dist-anc}}
\label{fig:dist-from-A}
\end{center}
\end{figure}

Finally,   we  define   $\mathfrak{T}(\ca)$, with the
metric $d$, as   the
completion of the metric space $(\mathfrak{T},d)$.

\begin{rem}
\label{rem:identif}
For every $i\in\ci^\text{d}$, we set $i_\ell$ the index in $\ci_0$ such that
$$x_{i_\ell}=\max\{x_j,\ 0\le x_j<x_i\mbox{ and }\zeta_j>\zeta_i\}.$$
Remark that $i_\ell$ is well defined since there are only a finite
number of indices $j\in\ci_0$ such that $x_j \in [0, x_i)$ and $\zeta_j>\zeta_i$. 
Similarly, for $i\in\ci^\text{g}$, we set $i_r$ the index in $\ci_0$ such that
$$x_{i_r}=\min\{x_j,\ x_i<x_j\le 0\mbox{ and }\zeta_j>\zeta_i\}.$$
The distance $d$ identifies the point $(x_i,-\zeta_i)$ (which does not
belong to $\mathfrak{T}$ by definition) with the point
$(x_{i_\ell},-\zeta_i)$ if $x_i>0$ and with the point
$(x_{i_r},-\zeta_i)$ if $x_i<0$ as illustrated on the right-hand
side of Figure \ref{fig:ancestral1}. 
\end{rem}

\begin{prop}
\label{proof:T(A)}
Let $\ca$ be an ancestral process. The tree
$({\mathfrak{T}}(\ca),d,(0,0))$ is a complete and locally compact rooted real tree with a unique semi-infinite branch and the associated forest belongs to $\T_1$.
\end{prop}

We shall call ${\mathfrak{T}}(\ca)$ the tree associated with the ancestral
process $\ca$. 

\begin{proof}
  By   construction    of   $(\mathfrak{T},d)$    see   \reff{eq:defTT},
  \reff{eq:dist-anc} and  \reff{eq:dist-anc2}, it is easy  to check that
  $\mathfrak{T}$ is connected 
%(as the union of an increasing sequence of  connected sets)
 and $d$ satisfies the so-called ``4-points condition''
  (see Lemma  3.12 in  \cite{e:prt}).  To conclude,  use that  those two
  conditions   characterize   real   trees    (see   Theorem   3.40   in
  \cite{e:prt}).  This gives  that $(\mathfrak{T},  d)$ as  well as  its
  completion are real trees.
By construction of $\mathfrak{T}$, it is easy to check that
$\mathfrak{T}(\ca)$ has a unique semi-infinite branch. \medskip

Let us now prove that $\mathfrak{T}(\ca)$ is locally compact. 
Let $(y_n, n\in \N)$ be a bounded sequence of
$\mathfrak{T}$. 

On  one  hand, let  us  assume  that there  exists  $i\in  \ci_0$ and  a
sub-sequence  $(y_{n_k},  k\in  \N)$  such  that  $y_{n_k}$  belongs  to
$S_i=\{x_i\}\times (-\zeta_i, 0]$.  Since,  for $i\in \ci$, there exists
a unique $j\in \ci_0$ such  that $S_i\cup \{(x_j,-\zeta_i)\}$ is compact
in  $(\mathfrak{T},d)$, see  Remark  \ref{rem:identif},  and for  $i=0$,
$S_0=\{0\}\times  (-\infty,0]$,  we  deduce that  the  bounded  sequence
$(y_{n_k},    k\in     \N)$    has    an    accumulation     point    in
$S_i\cup     \{(x_j,-\zeta_i)\}$     if     $i\in     \ci$     or     in
$\{0\}\times (-\infty , 0]$ if $i=0$.

On the  other hand,  let us assume  that for all  $i\in \ci_0$  the sets
$\{n, \,  y_n\in S_i\}$ are  finite. For  $n\in \N$, let  $i_n$ uniquely
defined  by $y_n\in  S_{i_n}$. Since  $(y_n,  n\in \N)$  is bounded,  we
deduce from Conditions (ii-iii) in Definition \ref{def:anc}, that the
sequence $(x_{i_n}, n\in  \N)$ is bounded in $\R$.  In particular, there
is  a sub-sequence  such that  $(x_{i_{n_k}}, k\in  \N)$ converges  to a
limit  say $a$.  Without  loss of  generality, we  can  assume that  the
sub-sequence  is  non-decreasing.  We  deduce  from  Condition  (ii)  in
Definition \reff{def:anc} that $\lim_{\varepsilon\downarrow 0}
  \max\{ \zeta_i, a-\varepsilon<x_i<a\}=0$. This implies, thanks to
  Definition \reff{eq:dist-anc}, that $(\{x_{i_{n_k}}\}\times\{0\}, k\in
  \N)$ is Cauchy in $\mathfrak{T}$ and using (ii) again that
  $\lim_{k\rightarrow+\infty } \zeta_{i_{n_k}}=0$. Then use that
\[
d(y_{n_k}, y_{n_{k'}})\leq  \zeta_{i_{n_k}} + \zeta_{i_{n_{k'}}}+ d
((x_{n_k}, 0), (x_{n_{k'}}, 0))
\]
to conclude that the $(y_{n_k}, k\in \N)$ is Cauchy in
$\mathfrak{T}$. 

We deduce that all bounded sequence  in $\mathfrak{T}$ has a Cauchy
sub-sequence.  This   proves  that  $\mathfrak{T}(\ca)$,  the
completion  of 
$\mathfrak{T}$ is locally compact.
\end{proof}

\begin{rem}
In the proof of Proposition \ref{proof:T(A)}, Conditions (i) and
(ii) in Definition \ref{def:anc} insure that
${\mathfrak{T}}(\ca)$ is a tree and  Conditions (ii) and  (iii) that 
${\mathfrak{T}}(\ca)$ is locally compact. 
\end{rem}

\subsection{The ancestral process of the Brownian forest}
\label{sec:apbf}

Let $\theta\geq 0$.  Let
$\cn(dh,d\varepsilon,de)=\sum_{i\in
  I}\delta_{(h_i,\varepsilon_i,e_i)}(dh,d\varepsilon,de)$ 
be, under $\P^ {(\theta)}$, a Poisson point measure on $\R\times\{-1,1\}\times\ce$ with intensity
$\beta    dh\,    (\delta_{-1}(d\varepsilon)+\delta_{1}(d\varepsilon))\,
n^{(\theta)}(de)$,
and  let  $\cf^{(\theta)} =((h_i,\tau_i), \, i\in  I)$ be  the  associated
Brownian forest where $\tau_i=T_{e_i}$ is the tree associated with the excursion $e_i$, see Section \ref{sec:CRT}. As  explained in Section \ref{sec:local_times}, this
Brownian forest models the evolution of a stationary population directed
by the branching mechanism $\psi_\theta$ defined in \reff{eq:psi}.

We want  to describe the genealogical  tree of the extant  population at
some  fixed time,  say 0.  The looked  after genealogical  tree is  then
$\cg_0(\cf^{(\theta)})$   defined   by   \reff{eq:def-G}.  To   describe   the
distribution of this  tree, we use an ancestral process  as described in
the  previous   subsection.  We   first  construct  a   contour  process
$(B_t,t\in\R)$ (obtained  by the  concatenation of two  independent Brownian motions
distributed as $B^{(\theta)}$)  which  codes  for the  tree  $\cf^{(\theta)}_{(-\infty,0]}$  (see
Section \ref{sec:forest} for the notations). The supplementary variables
$\varepsilon_i$ are needed  at this point to decide if  the tree $\bt_i$
is located on the left or on  the right of the infinite spine. The atoms
of the  ancestral process  are then  the pairs formed  by the  points of
growth of  the local time at  0 of $B$  and the depth of  the associated
excursion of $B$ below 0.

\subsubsection{Construction of the contour process}\label{sec:contour}
Let $\theta\geq 0$. Set $\ci=\{i\in I; \,  h_i<0\}$. 

For every $i\in \ci$, we set:
\[
a_i=\sum_{j\in \ci}\ind_{\{\varepsilon
  _j=\varepsilon_i\}}\ind_{\{h_j<h_i\}}\sigma(e_j)\quad\mbox{and}\quad
b_i=a_i+\sigma (e_i),
\]
where we  recall that  $\sigma(e_i)$ is the  length of  excursion $e_i$.
For every $t\ge 0$, we  set $i_t^\text{d}$ (resp.  $i_{t}^\text{g}$) the
only    index    $i\in\ci$    such    that    $\varepsilon_i=1$
(resp.  $\varepsilon_i=-1$) and  $a_i\le  t<b_i$.  Notice  that
$i_t^\text{d}$  and $i_{t}^\text{g}$ are a.s.~well defined but on a
Lebesgue-null set of values of $t$. 
We      set      $B^\text{d}=(B^\text{d}_t,     {t\ge      0})$      and
$B^\text{g}=(B^\text{g}_t, {t\ge 0})$ where for $t\geq 0$:
\[
B^\text{d}_t=h_{i_t^\text{d}}+e_{i_t^\text{d}}(t-a_{i_t^\text{d}})
\quad\mbox{and}\quad 
B^\text{g}_t=h_{i_{t}^\text{g} }+e_{i_{t}^\text{g}}(\sigma(e_{i_{t}^\text{g}})-(t-a_{i_{ t}^\text{g}})).
\]
We   deduce   from   Corollary  \ref{cor:b+-b-}   that   the   processes
$B^\mathrm{d}$ and  $B^\mathrm{g}$ are two independent  Brownian motions
distributed as $B^{(\theta)}$.
We define the process $B=(B_t,t\in\R)$ by
$B_t=B^\text{d}_t\ind_{\{t>0\}} + B^\text{g}_{-t}\ind_{\{t<0\}} $.
By construction,  the process  $B$ indeed  codes for  the tree
$\cf^{(\theta)}_{(-\infty,0]}$.

\subsubsection{The ancestral process}
 \label{sec:PPM}
Let   $(L_t^{\ell},t\ge  0)$ be the
local   time   at   0   of   the  process   $B^{\ell}$, where  $\ell\in\{\text{g},  \text{d}\}$.    We   denote   by
$((\alpha_i,\beta_i), \, i\in \ci^{\ell})$  the excursion intervals  of $B^{\ell}$
below 0, omitting the last  infinite excursion  if any,  and,  for every
$i\in                  \ci^{\ell}$,                 we                  set
$\zeta_i=-\min\{B_s^{\ell},\ s\in(\alpha_i,\beta_i)\}$.

We consider the  point measure on $\R\times \R_+$ defined by:
\[
\ca^\cn(du,d\zeta)=\sum_{i\in\ci^\text{d}}\delta_{(L^\text{d}_{\alpha_i},
  \zeta_i)}(du,d\zeta)+\sum_{i\in\ci^\text{g}}
\delta_{(-L^\text{g}_{\alpha_i},  \zeta_i)}(du,d\zeta).
\]

\begin{figure}[H]
\begin{center}
\psfrag{Bd}{$B_t^{d}$}
\psfrag{Bg}{$B^{g}_{-t}$}
\includegraphics[width=13cm]{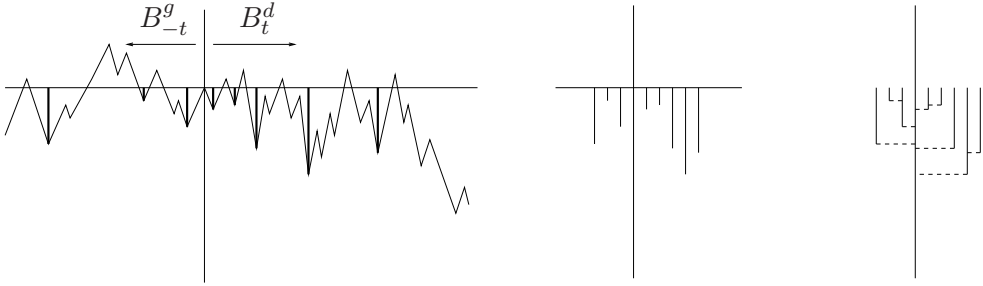}
\caption{The Brownian motions with drift, the ancestral process and the associated genealogical tree}\label{fig:ancestral1}
\end{center}
\end{figure}
See  Figure \ref{fig:ancestral1}  for  a representation  of the  contour
process $B$, the  ancestral process $\ca^\cn$ and  the genealogical tree
$\cg_0(\cf^{(\theta)})$.   In this figure,  the  horizontal  axis  represents the  time  for
Brownian motion  on the left-hand figure  whereas it is in  the scale of
local time for the ancestral process on the two right-hand figures. This
will always be the case in the  rest of the paper dealing with ancestral
processes.\medskip

Let $[-\eg,\ed]$ be the closed support of the measure
$\ca^\cn(du,\R_+)$:
\[
\ed  =\inf\{u\ge 0,\ \ca([u,+\infty)\times\R_+)=0\}
\quad \text{and}\quad
\eg  =\inf\{u\ge 0,\ \ca((-\infty,-u]\times\R_+)=0\},
\]
with the convention that $\inf \emptyset=+\infty $. Notice that, for $\ell\in\{\text{g},\text{d}\}$, we also have $E_\ell=L_\infty^\ell$.
We now give the distribution of the ancestral process $\ca^\cn$.
Recall  $c_\theta$  defined by \reff{eq:def-c}.

\begin{prop}\label{prop:ancestral}
  Let  $\theta\geq  0$.   Under $\P^{(\theta)}$,  the  random  variables
  $\eg,\ed$ are independent and exponentially distributed with parameter
  $2\theta$   (and   mean   $1/2\theta$)  with   the   convention   that
  $\ed=\eg=+\infty  $ if  $\theta=0$.  Under $\P^{(\theta)}$ and
  conditionally  given $(\eg,\ed)$, 
  the ancestral process $\ca^\cn(du, d\zeta)$ is a Poisson point measure
  with intensity:
\[
\ind_{(-\eg,\ed)}(u)\, du\, |c'_\theta(\zeta)|d\zeta. 
\]
\end{prop}

Notice that the random measure $\ca^\cn$ satisfies Conditions (i)-(iii)
from Definition \ref{def:anc} and is thus indeed an ancestral process. 

This result is very similar to Corollary 2 in \cite{bd:tl}. The main additional ingredient here is the order (given by the $u$ variable) which will be very useful in the simulation.

\begin{proof}
  Since  $B^\text{d}$ and  $B^\text{g}$  are independent  with the  same
  distribution, we deduce that $\eg$  and $\ed$ are independent with the
  same distribution.  Let $\theta>0$.   Since $B^\text{d}$ is a Brownian
  motion with drift $-2\theta$,  we deduce from Lemma \ref{lem:mb-drift}
  that $\ed$ is exponential with  mean $1/2\theta$.  The case $\theta=0$
  is immediate. 

  The excursions  below zero  of $B^\text{d}$ conditionally  given $E_d$
  are  excursions  of  a  Brownian  motion  $B^{(-\theta)}$  with  drift
  $2\theta$ (after  symmetry with respect  to $0$) conditioned  on being
  finite, that  is excursions of  a Brownian motion  $B^{(\theta)}$ with
  drift   $-2\theta$,  see   Lemma  \ref{lem:mb-drift}.    Moreover,  by
  \reff{eq:def-c}, $c_\theta$  is exactly  the tail distribution  of the
  maximum  of  an excursion  under  $n^{(\theta)}$.  Standard theory  of
  Brownian excursions gives then the result.
\end{proof}

\subsubsection{Identification of the trees}

Let   ${\mathfrak{T}}^\cn=\mathfrak{T}(\ca^\cn)$   denote  the   locally
compact  tree  associated  with  the ancestral  process  $\ca^\cn$,  see
Proposition \ref{proof:T(A)}. According to the following proposition, we
shall  say   that  the  ancestral   process  $\ca^\cn$  codes   for  the
genealogical tree  of the  extant population  at time  0 for  the forest
$\cf^{(\theta)}$.

\begin{prop}
   \label{prop:A=G}
 Let  $\theta\geq   0$. The trees $\cg_0(\cf^{(\theta)})$ under
   $\P^{(\theta)}$ and  ${\mathfrak{T}}^\cn$ belong to the same
   equivalence class in
   $\T_1$. 
\end{prop}

\begin{proof}
  Let us first remark that the genealogical tree $\cg_0(\cf^{(\theta)})$ can be
  directly  constructed using  the process  $B$ as  described on  Figure
  \ref{fig:ancestral2}.      
  
  More precisely, recall that $B$ is the contour function of the tree
  $\cf^{(\theta)}_{(-\infty,0]}$. Let us denote by $p_B$ the canonical projection
  from $\R$ to $\cf^{(\theta)}_{(-\infty,0]}$ as defined in Section
  \ref{sec:forest}. Recall $((\alpha_i,\beta_i), \, i\in \ci^{\ell})$,
  with   $\ell\in\{\text{g},  \text{d}\}$, are the excursion intervals
  of $B^\ell$ below $0$. 
Then $\cg_0(\cf^{(\theta)})$ is the smallest complete sub-tree
  of $\cf^{(\theta)}_{(-\infty,0]}$ that contains the points $(p_B(\alpha_i), i\in
  \ci_\text{g}\bigcup\ci_\text{d} )$ and the semi-infinite branch of $\cf^{(\theta)}_{(-\infty,0]}$. 
  
 %  More      precisely,      recall      that
%  $(\alpha_i,\beta_i)_{i\in \ci}$ denotes  the interval excursions below
%  0        of        the        process       $B$        and        that
%  $\zeta_i=-\min\{B_s,s\in(\alpha_i,\beta_i)\}$.     We     also     set
%  $E_\text{d}=\sup\{t>0,\                  B_t>0\}$                  and
%  $E_\text{g}=-\inf\{t<0,\ B_t>0\}$.
%
%We define a continuous function $g$ by setting
%\[
%\forall x\in (-E_\text{g},E_\text{d})\setminus
%\bigcup_{i\in\ci}(\alpha_i,\beta_i),\ g(x)=0,
%\]
%for all $i\in \ci$, $ g\left(\frac{\alpha_i+\beta_i}{2}\right)=-\zeta_i$,
% $ g$ is piecewise affine on $[\alpha_i,\beta_i]$, and
%\[
%\forall x\ge E_\text{d},\ g(x)=x-E_\text{d},
%\quad \forall x\le -E_\text{g}\ g(x)=x+E_\text{g}.
%\]
%Then, Using that $B$ is the contour
%process of the tree $\cf^\cn_{(-\infty , 0]}$, the function  $g$ can be viewed as the contour function of the tree $\cg_0(\cf^\cn)$.

\begin{figure}[H]
\begin{center}
\psfrag{Bd}{$B_t^{d}$}
\psfrag{Bg}{$B^{g}_{-t}$}
\includegraphics[width=9cm]{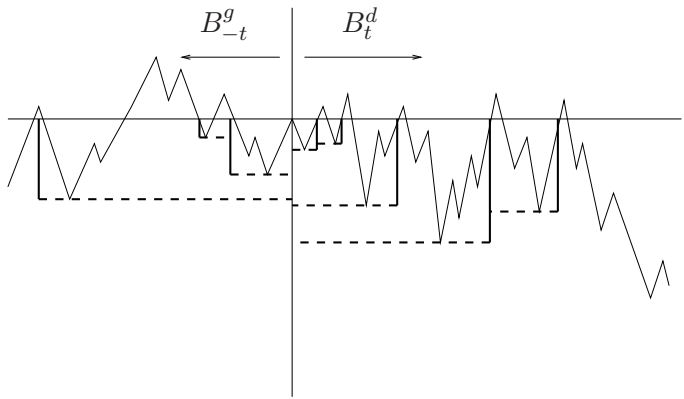}
\caption{The genealogical tree inside the Brownian motions}\label{fig:ancestral2}
\end{center}
\end{figure}

Let $i,j\in\ci$ with $0<\alpha_i<\alpha_j$  for instance.  By definition
of the  tree coded by  a function, the distance  between $p_B(\alpha_i)$
and $p_B(\alpha_j)$ in $\cg_0(\cf^{(\theta)})$ is given by:
\[
d(p_B(\alpha_i),p_B(\alpha_j))=-2\min_{u\in [\alpha_i,\alpha_j]}B_u.
\]
But, by definition of $\ca^\cn$, we have:
\begin{align*}
-\min_{u\in [\alpha_i,\alpha_j]}B_u & =\max_{k\in\ci\, \alpha_i\le \alpha_k< \alpha_j}\bigl(-\min_{u\in[\alpha_k,\beta_k]}B_u\bigr)\\
&=\max_{k\in\ci\, \alpha_i\le \alpha_k< \alpha_j} \zeta_k.
\end{align*}

The other cases $\alpha_j<\alpha_i<0$ and $\alpha _i<0<\alpha_j$ can be handled similarly. We deduce that the distances on a dense subset of leaves of $\cg_0(\cf^{(\theta)})$ and $\mathfrak{T}^\cn$ coincide, which implies the result by completeness of the trees.

\end{proof}

\subsubsection{Local times and other contour processes}\label{sec:local_times}
Recall $\theta\geq 0$. 
The
Brownian forest $\cf^{(\theta)}$ can be viewed as the genealogical tree of a stationary
continuous-state branching process (associated with the branching
mechanism $\psi_\theta$ defined in \reff{eq:psi}), see
\cite{cd:spsmrcatsbp}. To be more precise, for every $i\in I$ let 
$(\ell^{(i)}_a, \, a\ge 0)$ be the local time measures of the tree
$\tau_i$.
For every $t\in\R$, we consider  the measure $\rZ_t$ on $\cz_t(\cf^{(\theta)})$ defined by:
\begin{equation}
   \label{eq:rZdef}
\rZ_t(dx)=\sum _{i\in I} \ind_{\tau_i}(x)\, \ell_{t-h_i}^{(i)} (dx),
\end{equation}
and write  $Z_t=\rZ_t(1)$ for its  total mass which also  represents the
population  size at  time  $t$. For  $\theta=0$,  we have  $Z_t=+\infty$
a.s.~for every  $t\in\R$. For $\theta>0$, the process  $(Z_t, {t\ge 0})$
is a stationary Feller diffusion, solution of the SDE \reff{eq:SDE-Z}, see
Corollary 3.3. in \cite{cd:spsmrcatsbp}, see also Theorem 1.2 in \cite{d:crti}.  
\medskip

In the literature, one also consider  the so called Kesten tree which is
the genealogical tree associated with the Feller diffusion $Z^+$ solution
of \reff{eq:SDE-Z} for  $t\geq 0$ with initial condition  $ Z_0^+=0$. It
corresponds to the genealogy of a sub-critical branching process started
from an infinitesimal  individual alive at time 0,  conditionally on the
non-extinction event. In our setting, the genealogical tree correspond to
$\cf^{(\theta)}_{[0,  +\infty )}$ and the process $Z^+$ is distributed
as $\rZ^+(1)=(\rZ_t^+(1), t\geq 0)$ with the measure $\rZ^+_t$ on
$\cz_t(\cf^{(\theta)}_{[0,  +\infty )})$ defined by:
\[
\rZ^+_t(dx)=\sum _{i\in I} \ind_{\{h_i\geq 0\}}\, \ind_{\tau_i}(x)\, \ell_{t-h_i}^{(i)} (dx).
\]

It  can also  be  described  using a  contour  process  obtained by  the
concatenation   at  infinity   of  two   independent  Brownian   motions
distributed  as  $B^{(\theta)}$  conditioned  to be  positive.  We use  the
description given in  \cite{d:crti} which is valid in the general Lévy
case, see also Section  7.4 from  \cite{alw:ip}  which   corresponds
to the case 
$\theta=0$. 

Let $\ce_2$ be the set  of continuous non-negative functions $g$ defined
on         $\R$          such         that          $g(0)=0$         and
$\lim_{x\rightarrow  -\infty   }  g(x)=  \lim_{x\rightarrow   +\infty  }
g(x)=+\infty$.
For such a  function, we still consider the  pseudo-metric $d_g$ defined
by \reff{eq:dist_g} but  for $s,t\in\R$ and with  $m_g(s,t)$ replaced by
$m_g(s,t)= \inf_{r\not  \in [s\wedge  t,s\vee t ]}  g(r)$ if  $st<0$. We
define the  tree $T^+_g$  as the  quotient space on  $\R$ induced  by this
pseudo-metric.  We set $p_g$ as  the canonical projection from $\R$ onto
$T_g$.  For $g\in \ce_2$, the triplet $(T^+_g,d_g,p_g(0))$ is a complete
locally compact rooted real tree with a unique semi-infinite branch.
We still call $g$ the contour process of $T^+_g$.

Let $B^+=(B^+_t,     t\in     \R)$ be  such    that   $(B^+_t,    t\geq   0)$    and
$(B^+_{-t},   t\geq    0)$   are   independent   and    distributed   as
$B^{(\theta)}  -2 I^{(\theta)}$  which  is a  diffusion  on $\R_+$  with
infinitesimal generator given by \reff{eq:generator}. 
Thanks to  Corollary \ref{cor:b+-b-},  we get  that the tree $T^+_{B^+}$
with contour process $B^+$ is distributed as    the     genealogical     tree
$\cf^{(\theta)}_{[0,  +\infty  )}$ which is  associated to  the  Feller
diffusion   $\rZ^+(1)$  (solution
of \reff{eq:SDE-Z} on $\R^+$ with $Z_0=0$). 
\medskip 

Let $\theta>0$. It is also immediate  to give the contour process of the
genealogical  tree  conditionally  on   the  extinction  being  at  time
0. Recall the tree defined by its contour process with the concatenation
at 0 defined in Lemma \ref{lem:TginT1}. Set $B^-=-B^+$ It is left to the
reader to check that the tree  $T^-_{B^-}$ with contour process $B^-$ is
distributed   as  the   genealogical  tree   of  the   Feller  diffusion
$\rZ^-(1)=(\rZ^-(1)_t,  t\leq  0)$ conditioned  to  die  at time  0  and
started  with the  stationary distribution  at $-\infty  $ (solution  of
\reff{eq:SDE-Z} on $\R_-$ with $Z_0=0$),  where the measure $\rZ^-_t$ on
$\cz_t(\cf^{(\theta)}_{(-\infty , 0]})$ is defined by:
\[
\rZ^-_t(dx)=\sum _{i\in I} \ind_{\{\zeta_i+h_i< 0\}}\, \ind_{\tau_i}(x)\, \ell_{t-h_i}^{(i)} (dx),
\]
where $\zeta_i$ is the height of the tree $\tau_i$. This result can also
be deduced from the reversal property of the Brownian tree, see
\cite{ad:rpbt}.

\section{Simulation of the genealogical tree
  ($\theta>0$)}\label{sec:simul}

We use the representation of trees using the ancestral process, see Section
\ref{sec:ancestral}, which is an atomic measure on $\R^*\times (0,
+\infty )$ satisfying conditions of Definition \ref{def:anc}. 

Under                $\P^                {(\theta)}$,                let
$\sum_{i\in  I}\delta_{(h_i,\varepsilon_i,e_i)}$  be   a  Poisson  point
measure      on      $\R\times\{-1,1\}\times\ce$     with      intensity
$\beta    dh\,    (\delta_{-1}(d\varepsilon)+\delta_{1}(d\varepsilon))\,
n^{(\theta)}(de)$,
and  let $\cf^{(\theta)}=((h_i,\tau_i),  \,  i\in I)$  be  the associated  Brownian
forest.  We denote by $\ell_a^{(i)}$ the  local time measure of the tree
$\tau_i$  at  level  $a$  (recall  that this  local  time  is  zero  for
$a\not\in  [0,  H(\tau_i)]$) and we denote by $\partial _i$ the root of $\tau_i$.   Recall that the extant  population  at  time
$h\in   \R$    is   given    by   $\cz_h(\cf^{(\theta)})$   defined    in   Section
\ref{sec:genealogy} and the measure $\rZ_h$
on $\cz_h(\cf^{(\theta)})$ is defined by \reff{eq:rZdef}.

Let  $(\mathfrak{X}_k,  k\in \N^*)$  be,  conditionally  given $\cf^{(\theta)}$,  independent
random  variables  distributed  according  to  the  probability  measure
$\rZ_0/Z_0$. Remark that the normalization by $Z_0$, which is motivated
by the sampling approach,  is not usual in the branching setting, see
for instance Theorem 4.7 in \cite{cd:spsmrcatsbp}, where the sampling is
according to $\rZ_0$ instead leading to the bias factor $Z_0^n$. 

 For every $k\in \N^*$, we  set $i_k$ the index in $I$
such that $\mathfrak{X}_k\in\tau_{i_k}$.  For every $n\in\N^*$, we set $I_n=\{i_k,\ 1\le k\le n\}$ and for every $i\in I_n$, we denote by $\tau_i^{(n)}$ the sub-tree of $\tau_i$ generated by the $\mathfrak{X}_k$ such that $i_k=i$ and $1\le k\le n$, \textit{i.e.}:
\[
\tau_i^{(n)}=\bigcup_{1\le k\le n,\ i_k=i}\lb \partial_ i,\mathfrak{X}_k\rb.
\]
We define  the genealogical  tree $T_n$ of  $n$ individuals
sampled uniformly at random among the population at time 0 by:
\[
T_n=(-\infty , 0] \circledast_{i\in I_n}(\tau_i^{(n)},h_i).
\]
Notice  that   $T_n\subset  T_{n+1}$.   Since  the
support    of   $\rZ_h$    is   $\cz_h(\cf^{(\theta)})$ a.s.,    we   get    that   a.s.~$\clo \left(  \bigcup _{n\in \N^*} T_n  \right)= \cg_0(\cf^{(\theta)})$,
where $\cg_0(\cf^{(\theta)})$, see Definition \reff{eq:def-G}, is the genealogical
tree of the forest $\cf^{(\theta)}$ at time $0$.\medskip

Recall  $c_\theta$ defined by \reff{eq:def-c}.
For $\delta>0$, we will consider in the next sections a positive random
variable $\zeta_\delta^*$ whose distribution is given by, for $h>0$:
\begin{equation}
   \label{eq:dist-zd}
\P(\zeta^*_\delta<h)=\expp{-\delta c_\theta(h)}.
\end{equation}
This random variable is easy to simulate as, 
if $U$ is uniformly distributed on $[0,1]$, then  $\zeta^*_\delta$ has
the same distribution as:
\[
\inv{2\theta \beta} \log\left(1- \frac{2\theta\delta}{\log(U)}\right).
\]
This  random  variable  appears  naturally  in  the  simulation  of  the
ancestral         process         of        $\cf^{(\theta)}$         as,         if
$\sum_{i\in I}  \delta_{(z_i, \zeta_i)}$ is  a Poisson point  measure on
$\R\times\R_+$                       with                      intensity
$\ind_{   [0,  \delta]}(z)\,   dz   \,  |c'_\theta(\zeta)|d\zeta$   (see
Proposition   \ref{prop:ancestral}   for   the   interpretation),   then
$\zeta^*_\delta$ is distributed as $\max _{i\in I} \zeta_i$.

We now present many ways to simulate $T_n$. This will be done
by simulating ancestral processes, see Section \ref{sec:ancestral}, which
code for trees distributed as 
$T_n$. 

Recall that for an
  interval $I$, we write $|I|$ for its length. 

\subsection{Static simulation}
\label{sec:static}
In what  follows, $\text{S}$  stands for static.  Assume $n\in  \N^*$ is
fixed. We  present a  way to simulate  $T_n$ under  $\P^{(\theta)}$ with
$\theta>0$.   See  Figures  \ref{fig:unif}  and  \ref{fig:sim1}  for  an
illustration for $n=5$.

\begin{itemize}
\item[(i)] The size of the population on the left (resp. right) of the origin is $\eg$
  (resp.  $\ed$), where  $\eg, \ed$  are independent  exponential random
  variables with mean $1/2\theta$. Set  $Z_0=\eg+\ed$ for the total size
  of the population at time $0$.   Let $(U_k, k\in \N^*)$ be independent
  random variables uniformly distributed on  $[0, 1]$ and independent of
  $(\eg, \ed)$. Set $X_{0}=0$, and, for $k\in \N^*$, $X_k=Z_0 U_k - \eg$
  as well as $\cx_k=\{-\eg,\ed, X_0, \ldots, X_k\}$.
   \item[(ii)] 
For $1\leq k\leq n$, set $X_{k,n}^\text{g}=\max \{x\in
     \cx_n, \, x<X_k\}$ and  $X_{k,n}^\text{d}=\min \{x\in
     \cx_n, \, x>X_k\}$. We also set $I^\text{S}_{k}=[X_{k,n}^\text{g}, X_k]$ if
     $X_k>0$ and $I^\text{S}_{k}=[X_k, X_{k,n}^\text{d}]$ if
     $X_k<0$. 

   \item[(iii)]  Conditionally on  $(\eg, \ed,  X_1, \ldots,  X_n)$, let
     $(\zeta_{k}^{ \text{S}}, 1\leq  k\leq n)$  be independent  random variables
     such that  for $1\leq k\leq  n$, $\zeta_{k}^{ \text{S}}$ is  distributed as
     $\zeta^*_\delta$,        see         \reff{eq:dist-zd},        with
     $\delta= |I^\text{S}_{k}|$.   Consider the tree $\mathfrak{T}^\text{S}_n$  corresponding to
     the                        ancestral                        process
     $\ca_n^\text{S}=\sum_{k=1}^n \delta_{(X_k, \zeta_{k}^{ \text{S}})}$.

\end{itemize}

\begin{figure}[H]
\begin{center}
 \psfrag{Eg}{$-\eg$}
 \psfrag{Ed}{$\ed$}
 \psfrag{X0}{$X_0$}
 \psfrag{X1}{$X_1$}
 \psfrag{X2}{$X_2$}
 \psfrag{X3}{$X_3$}
 \psfrag{X4}{$X_4$}
 \psfrag{X5}{$X_5$}
\includegraphics[width=10cm]{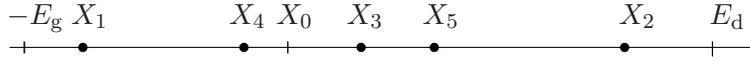}
\end{center}
\caption{One realization of $\eg, \ed, X_1, \ldots, X_5$.}
\label{fig:unif}
\end{figure}

\begin{figure}[H]
\begin{center}
 \psfrag{Eg}{$-\eg$}
 \psfrag{Ed}{$\ed$}
 \psfrag{X0}{$X_1$}
 \psfrag{X1}{$X_4$}
 \psfrag{X2}{$X_0$}
 \psfrag{X3}{$X_3$}
 \psfrag{X4}{$X_5$}
 \psfrag{X5}{$X_2$}
% \psfrag{Eg}{$-\eg$}
% \psfrag{Ed}{$\ed$}
% \psfrag{X0}{$X_{(0,5)}$}
% \psfrag{X1}{$X_{(1,5)}$}
% \psfrag{X2}{$X_{(2,5)}$}
% \psfrag{X3}{$X_{(3,5)}$}
% \psfrag{X4}{$X_{(4,5)}$}
% \psfrag{X5}{$X_{(5,5)}$}
\psfrag{I0}{$I^\text{S}_{4}$}
\psfrag{I5}{$I^\text{S}_{2}$}
\psfrag{z1}{$\zeta^\text{S}_{4}$}
\psfrag{z5}{$\zeta^\text{S}_{2}$}
% \psfrag{I0}{$\Delta_{(0,5)}$}
% \psfrag{I5}{$\Delta_{(5,5)}$}
% \psfrag{z1}{$\zeta_1$}
% \psfrag{z5}{$\zeta_5$}
\includegraphics[width=10cm]{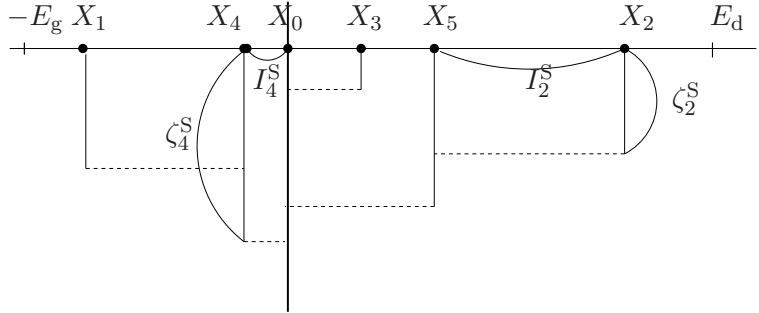}
\end{center}
\caption{One realization of  the tree $\mathfrak{T}_5^\text{S}$.}
\label{fig:sim1}
\end{figure}

This gives an exact simulation of the tree $T_n$ according to the following result.

\begin{lem}
   \label{lem:arbre-n}
Let $\theta>0$ and  $n\in \N^*$. The tree $\mathfrak{T}_n^\text{S}$ is distributed as
$T_n$ under $\P^{(\theta)}$. 
\end{lem}

\begin{proof}
Let $B=(B_t, t\in \R)$ be the Brownian motion with drift defined in Section \ref{sec:contour} and let $(L_t,t\in\R)$ be its local time at 0 \textit{i.e.}:
\[
L_t=L_t^{\text{d}}\ind_{t>0}+L_{-t}^{\text{g}}\ind_{t<0}.
\]
We set $L_\infty=L_\infty^{\text{d}}+L_{\infty}^{\text{g}}$ and we consider i.i.d. variables $(S_1,\ldots, S_n)$ distributed according to $dL_s/L_\infty$. We denote by $(S_{(1)},\ldots,S_{(n)})$ the order statistics of $(S_1,\ldots S_n)$ and, for every $i\le n$, we set
\[
M_i=\begin{cases}
-\min_{u\in [S_{(i)}, S_{(i+1)}\wedge 0]}B_u & \mbox{if }S_{(i)}<0,\\
-\min_{u\in [S_{(i-1)}\vee 0, S_{(i)}]}B_u & \mbox{if }S_{(i)}>0.
\end{cases}
\]

We set $\mathcal{A}_n=\sum_{1\le i\le n}\delta_{(L_{S_{(i)}},\zeta_i)}$
which is (see Definition \ref{def:anc}) an ancestral process and let
$\mathfrak{T}(\mathcal{A}_n)$ be the associated tree. As $B$ is the
contour process of the tree $\mathcal{F}_{(-\infty,0]}$, we get that
$T_n$ and $\mathfrak{T}(\mathcal{A}_n)$ are equally distributed. 

Moreover, by Proposition \ref{prop:ancestral}, Proposition
\ref{prop:A=G} and standard results on Poisson point processes, we get
that $\mathfrak{T}(\mathcal{A}_n)$ and $\mathfrak{T}_n^S$ are also
equally distributed. 
\end{proof}

\subsection{Dynamic simulation (I)}\label{sec:dynamic1}

We can modify the static simulation of the previous section to provide a
natural dynamic construction of the  genealogical tree. In what follows,
$\text{D}$ stands for  dynamic.  Let $\theta>0$.  We  build recursively a
family    of   ancestral    processes   $(\ca_n,    n\in   \N)$,    with
$\ca_0^\text{D}=0$                                                   and
$\ca_n^\text{D}=\sum_{k=1}^n   \delta_{(V_k,   \zeta_k^\text{D})}$   for
$n\in \N^*$.

\begin{itemize}
\item[(i)] Let $\eg$,  $\ed$, $(X_n, n\in \N)$ and  $(\cx_n, n\in \N^*)$
  be defined  as in (i) of  Section \ref{sec:static}.   For  $n\in \N  ^*$, set
  $X_n^\text{g}=\max       \{x\in        \cx_n,       x<X_n\}$       and
  $X_n^\text{d}=\min  \{x\in  \cx_n,  x>X_n\}$. 
  
   For $n\in  \N  ^*$  and
  $\ell\in    \{\text{g},     \text{d}\}$,    define     the    interval
  $I_n^{\ell}= [X_n  \wedge X_n^\ell, X_n\vee X_n^\ell]$ and its length $|
I_n^{\ell}|=|X_n - X_n^\ell|$. 

We shall consider  and check by the induction  the following hypothesis:
for  $n\geq 2$  the random  variables $V_1,  \ldots, V_{n-1}$ are such that
\begin{equation}
   \label{eq:entrelace}
X_{(0,n-1)}<V_{(1,n-1)}<X_{(1, n-1)}< \cdots < V_{(n-1, n-1)}< X_{(n-1,
  n-1)},
\end{equation}
where      $(V_{(1,      n-1)},     \ldots,      V_{(n,n)})$   and
$(X_{(0,n-1)},  \ldots,  X_{(n-1,n-1)})$   respectively  are  the  order
statistics  of $(V_1,  \ldots,  V_{n-1})$ and  of $(X_0, \ldots,  X_{n-1})$
respectively. Notice that \reff{eq:entrelace} holds trivially for $n=1$. 

We set $\cw_n^\text{D}=(\eg, \ed,  X_1, \ldots, X_n,  V_1, \ldots, V_{n-1},
\zeta^\text{D}_{1}, \ldots, \zeta^\text{D}_{n-1})$.

\item[(ii)]    Assume   $n\geq    1$.   We    work   conditionally    on
  $\cw^\text{D}_n$.  On  the event
  $\{X_n^\text{d}=\ed\}$  set   $I_n=I_n^\text{g}$  and  on   the  event
  $\{X_n^\text{g}=-\eg\}$   set   $I_n=I_n^\text{d}$.     On   the   event
  $\{X_n^\text{d}=\ed\}\bigcup  \{X_n^\text{g}=-\eg\}$,   let  $V_n$  be
  uniform   on    $I_n$   and   $\zeta_{n}^\text{D}$   be    distributed   as
  $\zeta^*_\delta$, see \reff{eq:dist-zd}, with $\delta= |I_n|$. 

\item[(iii)] Assume  $n\geq 2$  and that \reff{eq:entrelace}  holds.  We
  work     conditionally     on     $\cw^\text{D}_n$.      On     the     event
  $\{-\eg<X_n^\text{g},\,X_n^\text{d}<\ed\}$,  there   exists  a  unique
  integer    $\kappa_n\in     \{1,    \ldots,    n-1\}$     such    that
  $V_{\kappa_n}\in       [      X_n^\text{g},\,X_n^\text{d}]$.        If
  $X_n\in [X_n^\text{g}, V_{\kappa_n})$,  set $I_n=I_n^\text{g}$; and if
  $X_n\in [V_{\kappa_n}, X_n^\text{d}]$, set $I_n=I_n^\text{d}$.  Then, let $V_n$  be uniform
  on $I_n$ and $\zeta_{n}^\text{D}$  be distributed as $\zeta^*_\delta$,
  with   $\delta=    |I_n|$,   conditionally   on   being    less   than
  $\zeta^\text{D}_{\kappa_n}$.

\item[(iv)] Thanks  to (ii)  and (iii), notice  that \reff{eq:entrelace}
  holds with $n-1$ replaced by $n$,  so that the induction is valid. Set
  $\ca_n^\text{D}=\ca_{n-1}^\text{D}+                      \delta_{(V_n,
    \zeta^\text{D}_{n})}$
  and consider  the tree $\mathfrak{T}_n^\text{D}$ corresponding  to the
  ancestral process $\ca_n^\text{D}$.
\end{itemize}

See Figures \ref{fig:D1-4} and \ref{fig:D1-5} for an instance of
$\mathfrak{T}_4^\text{D}$ and $\mathfrak{T}_5^\text{D}$.

\begin{figure}[H]
\begin{center}
\psfrag{Eg}{$-\eg$}
\psfrag{Ed}{$\ed$}
\psfrag{X0}{$X_{1}$}
\psfrag{X1}{$X_{4}$}
\psfrag{X2}{$X_{0}$}
\psfrag{X3}{$X_{3}$}
\psfrag{X4}{$X_{2}$}
\psfrag{V1}{$V_{4}$}
\psfrag{V2}{$V_{1}$}
\psfrag{V3}{$V_{3}$}
\psfrag{V4}{$V_{2}$}
\includegraphics[width=10cm]{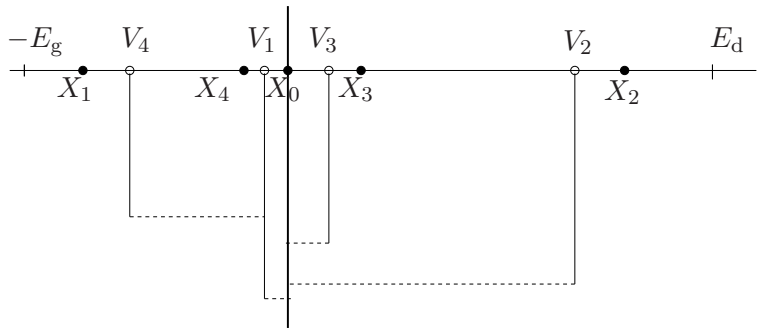}
\end{center}
\caption{An instance of the tree $\mathfrak{T}_4^\text{D}$.}
\label{fig:D1-4}
\end{figure}

\begin{figure}[H]
\begin{center}
\psfrag{Eg}{$-\eg$}
\psfrag{Ed}{$\ed$}
\psfrag{X0}{$X_{1}$}
\psfrag{X1}{$X_{4}$}
\psfrag{X2}{$X_{0}$}
\psfrag{X3}{$X_{3}$}
\psfrag{X4}{$X_{5}$}
\psfrag{X5}{$X_{2}$}
\psfrag{V1}{$V_{4}$}
\psfrag{V2}{$V_{1}$}
\psfrag{V3}{$V_{3}$}
\psfrag{V4}{$V_{5}$}
\psfrag{V5}{$V_{2}$}
\includegraphics[width=10cm]{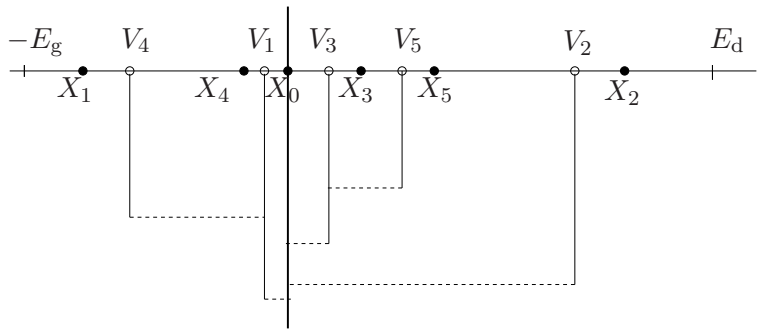}
\end{center}
\caption{An instance of the tree $\mathfrak{T}_5^\text{D}$. The length
  of the new branch attached to $V_{5}$ is conditioned to be less than
  the previous branch that was in the considered interval 
  attached to $V_{2}$} 
\label{fig:D1-5}
\end{figure}

Then we have the following result. 

\begin{lem}
   \label{lem:arbre-n-dynamic1}
Let $\theta>0$. The sequences of trees $(\mathfrak{T}_n^\text{D}, n\in \N^*)$ and
$(T_n, n\in \N^*)$ under $\P^{(\theta)}$ have
the same distribution. 
\end{lem}

\begin{proof}
  We  consider  $\sum  _{i\in  \ci}\delta_{(u_i,\zeta_i)}$  the  ancestral
  process     associated     to     the    Poisson     point     measure
  $\sum_{i\in  I}\delta_{(h_i,\varepsilon_i,e_i)}$  defined  in  Section
  \ref{sec:PPM}.   Let  $(X''_k, k\in  \N  ^*)$  be independent  uniform
  random variables on  $[-\eg, \ed]$. Set $X''_0=0$.  For  $n\ge 1$, let
  us denote  by $(X''_{(k,n)},  {0\le k\le n})$  the order  statistic of
  $(X''_0,\ldots,X''_n)$.

For every $n\ge 1$ and every $1\le k\le n$, we set $i_{k,n}$ the index in $\ci$ such that
\[
\zeta_{i_{k,n}}=\max _{X''_{(k-1,n)}\le u_i< X''_{(k,n)}}\zeta_i.
\]
Remark that this index exists  since, for every $\varepsilon>0$, the set
$\{i\in \ci,\ \zeta _i>\varepsilon\}$ is  a.s.~finite.  We set
$V''_{(k,n)}=u_{i_{k,n}}$ and notice that, by standard Poisson point measure properties, $V''_{(k,n)}$ is, conditionally given $(X''_0,\ldots,X''_n)$, uniformly distributed on $[X''_{(k-1,n)},X''_{(k,n)}]$. We define
\[
\ca''_n=\sum_{k=1}^n \delta_{(V''_{(k,n)}, \zeta_{i_{k,n}})}.
\]

By construction, it is easy to check that the order statistics
\[
X''_{(0,n)}<V''_{(1,n)}<X''_{(1, n)}< \cdots < V''_{(n, n)}< X''_{(n,
  n)}
\]
is distributed as 
\[
X_{(0,n)}<V_{(1,n)}<X_{(1, n)}< \cdots < V_{(n, n)}< X_{(n,
  n)}.
\]
For $1\leq k\leq n$, let $j_{k,n}\in \{1, \ldots, n\}$ be the index such
that $V_{(k,n)}=V_{j_{k,n}}$. 
By construction, we then deduce that $(((V_{(k,n)},
\zeta^\text{D}_{j_{k,n}}), 1\leq k\leq n), n\in \N^*)$ is distributed as $(((V''_{(k,n)},
\zeta_{i_{k,n}}), 1\leq k\leq n), n\in \N^*)$. This implies that the
sequence of ancestral processes $(\ca''_n, n\in \N^*)$ and $(\ca_n, n\in
\N^*)$ have the same distribution. Then use Proposition \ref{prop:A=G}
to get that the sequence of trees $(T''_n, n\in \N^*)$, with $T''_n$
associated to $\ca''_n$, is distributed as  
$(T_n, n\in \N^*)$. 
\end{proof}

\subsection{Dynamic simulation (II)}\label{sec:dynamic2}

In a sense, we had to introduce another random information corresponding
to the position $V_{n}$ of the  largest spine of the sub-tree containing
$X_{n}$. The construction  in this sub-section provides a  way to remove
this additional information (which is now  hidden) but at the expense to
possibly exchange the  new inserted branch with one of  its neighbor. In
what follows, $\text{H}$ stands for  hidden. An instance is provided for
$\mathfrak{T}^\text{H}_4$   and  $\mathfrak{T}^\text{H}_5$   in  Figures
\ref{fig:T4}, \ref{fig:T5g} and \ref{fig:T5d}.

Let $\theta>0$.
We build recursively a family of ancestral processes $(\ca_n^\text{H}, n\in \N)$,
with $\ca_0^\text{H}=0$ and  $\ca_n^\text{H}=\sum_{k=1}^n \delta_{(X_k, \zeta_{k,n}^\text{H})}$ for
$n\in \N^*$.

\begin{itemize}
\item[(i)] Let $\eg$,  $\ed$, $(X_n, n\in \N)$ and  $(\cx_n, n\in \N^*)$
  be defined  as in (i) of  Section \ref{sec:static}.   For  $n\in \N  ^*$, set
  $X_n^\text{g}=\max       \{x\in        \cx_n,       x<X_n\}$       and
  $X_n^\text{d}=\min  \{x\in  \cx_n,  x>X_n\}$.  For $n\in  \N  ^*$  and
  $\ell\in    \{\text{g},     \text{d}\}$,    define     the    interval
  $I_n^{\ell}= [X_n  \wedge X_n^\ell, X_n\vee X_n^\ell]$ and its length $|
I_n^{\ell}|=|X_n - X_n^\ell|$. 

We set $\cw_n^\text{H}=(\eg, \ed,  X_1, \ldots, X_n,  
\zeta^\text{H}_{1, n-1}, \ldots, \zeta^\text{H}_{n-1, n-1})$.

\item[(ii)] Assume  $n\geq 1$.  On the  event $\{X_n^\text{d}=\ed\}$ set
  $I_n=I_n^\text{g}$  and  on   the  event  $\{X_n^\text{g}=-\eg\}$  set
  $I_n=I_n^\text{d}$.  Conditionally on $\cw_n^\text{H}$,
  let  $\zeta_{n,n}^\text{H}$  be   distributed  as  $\zeta^*_\delta$,  see
  \reff{eq:dist-zd}, with  $\delta= |I_n|$;  and for $1\leq  k\leq n-1$,
  set $\zeta^\text{H}_{k,n}=\zeta^\text{H}_{k,n-1}$.

\item[(iii)]

  Assume $n\geq  2$.  We  work conditionally  on $\cw_n^\text{H}$.  
 We define:
\[
\pd=\frac{|I_n^\text{d}|}{|I_n^\text{d}|+|I_n^\text{g}|}
\quad\text{and}\quad
\pg=1-\pd=\frac{|I_n^\text{g}|}{|I_n^\text{d}|+|I_n^\text{g}|} 
 \cdot
\]

 \begin{itemize}
   \item[(a)] On the
  event  $\{0\leq  X_n^\text{g},\,X_n^\text{d}<\ed\}$,  there  exists  a
  unique integer  $\kappa_n^\text{d} \in  \{1, \ldots, n-1\}$  such that
  $X_{\kappa_n^\text{d}}  = X_n^\text{d}$.   For $1\leq  k\leq n-1$  and
  $k\neq                     \kappa_n^\text{d}$,                     set
  $\zeta^\text{H}_{n,k}        =         \zeta^\text{H}_{n-1,k}$.        Write
  $\zeta^\text{H}_n=\zeta^\text{H}_{n-1, \kappa_n^\text{d}}$. 

With probability $\pd$, set  $\zeta_{n,\kappa_n^\text{d}}^\text{H}= \zeta^\text{H}_{n}$ and 
let  $\zeta_{n,n}^\text{H}$ be distributed  as $\zeta^*_\delta$, with   $\delta=
  |I_n^\text{g}|$,     conditionally     on    being     less     than
  $\zeta^\text{H}_{n}$.

With probability $\pg$, set  $\zeta_{n,n}^\text{H}= \zeta^\text{H}_{n}$ and
let $\zeta^\text{H}_{n,
  \kappa_n^\text{d}}$ be distributed  as $\zeta^*_\delta$, with   $\delta=
  |I_n^\text{d}|$,     conditionally     on    being     less     than
  $\zeta^\text{H}_{n}$.

\item[(b)] On the
  event  $\{-\eg<  X_n^\text{g},\,X_n^\text{d}\leq 0\}$,  there  exists  a
  unique integer  $\kappa_n^\text{g} \in  \{1, \ldots, n-1\}$  such that
  $X_{\kappa_n^\text{g}}  = X_n^\text{g}$.   For $1\leq  k\leq n-1$  and
  $k\neq                     \kappa_n^\text{g}$,                     set
  $\zeta^\text{H}_{n,k}        =         \zeta^\text{H}_{n-1,k}$.        Write
  $\zeta^\text{H}_n=\zeta^\text{H}_{n-1, \kappa_n^\text{g}}$. 

With probability $\pg$, set  $\zeta_{n,\kappa_n^\text{g}}^\text{H}= \zeta^\text{H}_{n}$ and 
let  $\zeta_{n,n}^\text{H}$ be distributed  as $\zeta^*_\delta$, with   $\delta=
  |I_n^\text{d}|$,     conditionally     on    being     less     than
  $\zeta^\text{H}_{n}$.

With probability $\pd$, set  $\zeta_{n,n}^\text{H}= \zeta^\text{H}_{n}$ and
let $\zeta^\text{H}_{n,
  \kappa_n^\text{g}}$ be distributed  as $\zeta^*_\delta$, with   $\delta=
  |I_n^\text{g}|$,     conditionally     on    being     less     than
  $\zeta^\text{H}_{n}$.
\end{itemize}

\item[(iv)]  Let  $\mathfrak{T}^\text{H}_n$ be the tree corresponding to the ancestral process   $\ca_n^\text{H}=\sum_{k=1}^n \delta_{(X_k,  \zeta^\text{H}_{k,n})}$.
\end{itemize}

We have the next result. 

\begin{lem}
   \label{lem:arbre-n-dymanic2}
   Let       $\theta>0$.        The       sequences       of       trees
   $(\mathfrak{T}^\text{H}_n,  n\in  \N^*)$  and $(T_n,  n\in  \N^*)$  under
   $\P^{(\theta)}$ have the same distribution.
\end{lem}

\begin{proof}
The proof is left to the reader. It is in the same spirit as the proof of Lemma
\ref{lem:arbre-n-dynamic1}, but here we consider the random variables
$((V''_{(k,n)}, 1\leq k\leq n), n\in \N^*)$ as unobserved. 
\end{proof}

\begin{figure}[H]
\begin{center}
\psfrag{Eg}{$-\eg$}
\psfrag{Ed}{$\ed$}
\psfrag{X0}{$X_{1}$}
\psfrag{X1}{$X_{4}$}
\psfrag{X2}{$X_{0}$}
\psfrag{X3}{$X_{3}$}
\psfrag{X4}{$X_{2}$}
 \psfrag{X5}{$X_5$}
 \psfrag{z2}{$\zeta_{4,2}^\text{H}$}
% \psfrag{V1}{$V_{1,4}$}
% \psfrag{V2}{$V_{2,4}$}
% \psfrag{V3}{$V_{3,4}$}
% \psfrag{V4}{$V_{4,4}$}
\includegraphics[width=10cm]{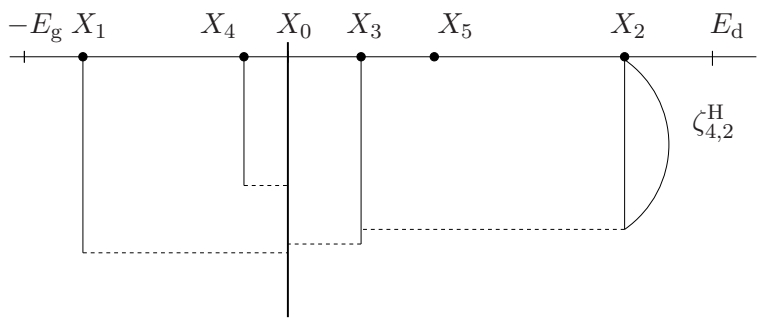}
\end{center}
\caption{An instance of the tree $\mathfrak{T}^\text{H}_4$ with the new individual $X_5$.}
\label{fig:T4}
\end{figure}

\begin{figure}[H]
\begin{center}
\psfrag{Eg}{$-\eg$}
\psfrag{Ed}{$\ed$}
\psfrag{X0}{$X_{1}$}
\psfrag{X1}{$X_{4}$}
\psfrag{X2}{$X_{0}$}
\psfrag{X3}{$X_{3}$}
\psfrag{X4}{$X_{5}$}
 \psfrag{X5}{$X_2$}
 \psfrag{z2}{$\zeta_{5,2}^\text{H}$}
% \psfrag{Eg}{$-\eg$}
% \psfrag{Ed}{$\ed$}
% \psfrag{X0}{$X_{(0,5)}$}
% \psfrag{X1}{$X_{(1,5)}$}
% \psfrag{X2}{$X_{(2,5)}$}
% \psfrag{X3}{$X_{(3,5)}$}
% \psfrag{X4}{$X_{(4,5)}$}
% \psfrag{X5}{$X_{(5,5)}$}
% \psfrag{V1}{$V_{1,5}$}
% \psfrag{V2}{$V_{2,5}$}
% \psfrag{V3}{$V_{3,5}$}
% \psfrag{V4}{$V_{4,5}$}
% \psfrag{V5}{$V_{5,5}$}
\includegraphics[width=10cm]{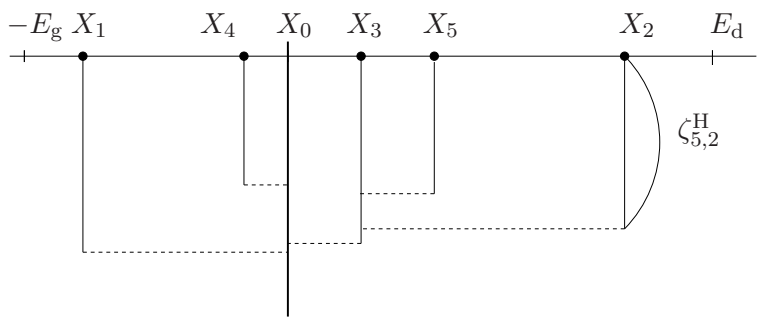}
\end{center}
\caption{An instance of the tree $\mathfrak{T}^\text{H}_5$  with
  $\mathfrak{T}^\text{H}_4$ given in Figure 
  \ref{fig:T4} and the event associated with $\pd$ (a new segment is
  attached to $X_{5}$).}
\label{fig:T5g}
\end{figure}

\begin{figure}[H]
\begin{center}
\psfrag{Eg}{$-\eg$}
\psfrag{Ed}{$\ed$}
\psfrag{X0}{$X_{1}$}
\psfrag{X1}{$X_{4}$}
\psfrag{X2}{$X_{0}$}
\psfrag{X3}{$X_{3}$}
\psfrag{X4}{$X_{5}$}
 \psfrag{X5}{$X_2$}
  \psfrag{z2}{$\zeta_{5,2}^\text{H}$}
% \psfrag{Eg}{$-\eg$}
% \psfrag{Ed}{$\ed$}
% \psfrag{X0}{$X_{(0,5)}$}
% \psfrag{X1}{$X_{(1,5)}$}
% \psfrag{X2}{$X_{(2,5)}$}
% \psfrag{X3}{$X_{(3,5)}$}
% \psfrag{X4}{$X_{(4,5)}$}
% \psfrag{X5}{$X_{(5,5)}$}
% \psfrag{V1}{$V_{1,5}$}
% \psfrag{V2}{$V_{2,5}$}
% \psfrag{V3}{$V_{3,5}$}
% \psfrag{V4}{$V_{4,5}$}
% \psfrag{V5}{$V_{5,5}$}
\includegraphics[width=10cm]{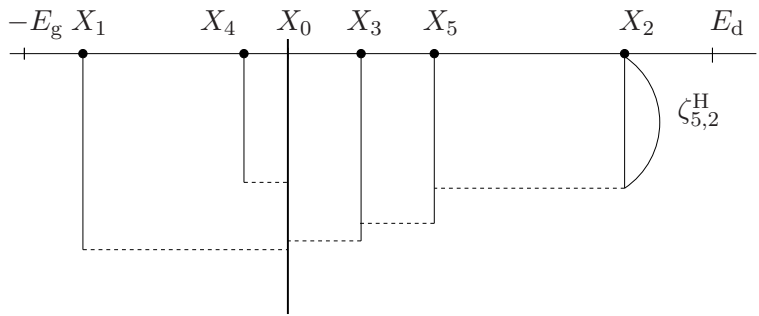}
\end{center}
\caption{An instance of the tree $\mathfrak{T}^\text{H}_5$  with
  $\mathfrak{T}^\text{H}_4$ given in Figure 
  \ref{fig:T4} and the event associated with $\pg$ (the segment
  previously attached to $X_2$ is now attached to $X_5$ and a new
  segment is attached to $X_2$).}
\label{fig:T5d}
\end{figure}

\subsection{Simulation of genealogical tree conditionally on its maximal
height}\label{sec:cond_simul}
Let $\cf^{(\theta)}=((\tau_i, h_i), \, i\in I)$ be a Brownian forest under
$\P^{(\theta)}$. Recall the definition of $A_0$ the time to the MRCA of
the population living at time $0$ given in \reff{eq:def-Ah2}. 
The goal of this section is
to  simulate  the  genealogical  tree  $T_n$  of  $n$  
individuals uniformly sampled in the population living at time $0$,
conditionally given the time to the MRCA of the whole population is $h$,
that is given  $A_0=h$.  \medskip

Let $\ca(du,d\zeta)=\sum_{j\in \ci}\delta_{(u_j,\zeta_j)}(du,d\zeta)$ 
be the  ancestral process of Definition  \ref{def:anc}. Recall the
notations    $\eg,\ed$    from     Sectionj   \ref{sec:PPM}.     Let
$\zeta_\text{max}=\sup  \{\zeta_j,\ j\in  \ci\}$ and  define the  random
index $J_0\in  \ci$ such that $\zeta_\text{max}=\zeta_{J_0}$.  Note that
$J_0$  is  well  defined  since   for  every  $\varepsilon>0$,  the  set
$\{j\in \ci, \zeta_j>\varepsilon\}$ is finite.  We set $X=u_{J_0}\in
(-\eg, \ed)$.
Remark  that  $\zeta_\text{max} $ is distributed as $A_0$. 

For $r\in \R$, let $r_+=\max(0,r)$ and $r_-=\max(0, -r)$ be respectively
the positive and negative part of $r$. 
The proof of the next lemma is postponed to the end of this section. 
\begin{lem}
\label{lem:EEX}
Let $\theta>0$. Under $\P^{(\theta)}$, conditionally given $\zeta_\text{max}=h$, the random
variables $E_g+X_-$, $|X|$, $E_d-X_+$ and $\ind_{\{X\geq 0\}}$ are independent;  $E_g+X_-$, $|X|$ and $E_d-X_+$
are exponentially distributed with parameter $2\theta +c_\theta(h)$ and
$\ind_{\{X\geq 0\}}$ is Bernoulli $1/2$. 
\end{lem}

Let $h>0$ be fixed. 
For $\delta>0$, let $\zeta^{*,h}_\delta$ be a positive random variable
distributed as $\zeta^*_\delta$ conditionally on $\{\zeta^*_\delta\leq
h\}$,  \textit{i.e.}~for $0\le u\le h$:
\[
\P(\zeta^{*,h}_\delta \le u)=\P(\zeta_\delta^*\le
u\bigm|\zeta_\delta^*\le h)=\expp{-\delta(c_\theta(u)-c_\theta(h))}.
\]
Then the static simulation runs as follows.
\begin{itemize}
\item[(i)]  Simulate three  independent  random variables  $E_1,E_2,E_3$
  exponentially  distributed with  parameter $2\theta+c_\theta(h)$,  and
  another independent Bernoulli variable $\xi$ with parameter $1/2$.  If
  $\xi=0$,  set $\eg=E_1,\  X=E_2,\  \ed=E_2+E_3$, and  if $\xi=1$,  set
  $\eg=E_1+E_2,\ X=-E_2,\ \ed=E_3$. Let $X_k$  and $\cx_k$ be defined as
  in (i) of Section \ref{sec:static} for $1\leq k\leq n$.

\item[(ii)] Let the intervals $I_k^\text{S}$ be defined as in (ii) of Section
  \ref{sec:static} for $1\leq k\leq n$. 

\item[(iii)] Conditionally on $(\eg, \ed,  X, X_1, \ldots, X_n)$, let 
  $(\zeta_{k}^{h}, 1\leq k\leq n)$  be independent random variables such
  that, for $1\leq k\leq n$,  $\zeta_{k}^{h}$ is distributed as
  $\zeta^{*,h}_\delta$     with     $\delta=    |I^\text{S}_{k}|$     if
  $X    \not    \in     I_k^\text{S}$;    and    $\zeta_{k}^{h}=h$    if
  $X   \in   I_k^\text{S}$.    Consider  the   tree   $\mathfrak{T}^h_n$
  corresponding         to         the         ancestral         process
  $\ca_n^h=\sum_{k=1} \delta_{(X_k, \zeta_{k}^{h})}$.

\end{itemize}

The proof of the following result which relies on Lemma \ref{lem:EEX} is similar
to the one of Lemma~\ref{lem:arbre-n}, and is not reproduced here.
\begin{lem}
   \label{lem:arbre-h}
Let $\theta>0$, $h>0$  and  $n\in \N^*$. The tree $\mathfrak{T}_n^h$ is distributed as
$T_n$ under $\P^{(\theta)}$ conditionally given $A_0=h$. 
\end{lem}

Notice that the height of $\mathfrak{T}_n^h$ is less than or equal to
$h$. When strictly less than $h$, it means that no individual of the
oldest family has been sampled. 

\begin{proof}[Proof of Lemma \ref{lem:EEX}]
By Proposition \ref{prop:ancestral}, the pair $E=(\eg,\ed)$ under
$\P^{(\theta)}$ has density: 
\[
f_E(\eeg, \eed)=(2\theta)^2\expp{-2\theta(\eeg+\eed)}\ind_{\{\eeg\ge 0,
  \eed\ge 0\}}.
\]
Moreover, by standard results on Poisson point measures, the conditional
density  of the pair $(X,\zeta_\text{max})$ given $(\ed,\eg)=(e_g,e_d)$
exists and is:
\begin{align*}
f_{X,\zeta_\text{max}}^{E=(\eeg,\eed)}(x,h) 
&=\frac{1}{\eeg+\eed}\ind_{[-\eeg,\eed]}(x)\,
 (\eeg+\eed)\, |c'_\theta(h)|\expp{-c_\theta(h)(\eeg+\eed)}\ind_{\{h\ge 0\}}\\
& =\ind_{[-\eeg,\eed]}(x)\,
  |c'_\theta(h)|\expp{-c_\theta(h)(\eeg+\eed)}\ind_{\{h\ge 0\}}. 
\end{align*}
We deduce that the vector $(\eg,\ed,X,\zeta_\text{max})$ has density:
\[
f(\eeg,\eed,x,h)=(2\theta)^2|c'_\theta(h)|\expp{-(2\theta+c_\theta(h))(\eeg+\eed)}\ind_{\{  
  \eeg\ge 0, \, \eed\ge 0, \, -\eeg\le x\le \eed, \, h\ge 0\}}
\]
and that the random variable $\zeta_\text{max}$ has density:
\begin{align*}
f_{\zeta_\text{max}}(h) 
& =\int (2\theta)^2|c'_\theta(h)|\expp{-(2\theta+c_\theta(h))(\eeg+\eed)}
\ind_{\{
  \eeg\ge 0, \, \eed\ge 0, \, -\eeg\le x\le \eed, \, h\ge 0\}}
\,d\eeg\,d\eed\,dx\\
& =(2\theta)^2|c'_\theta(h)|\frac{2}{(2\theta+c_\theta(h))^3}\, \ind_{\{h\ge 0\}}.
\end{align*}
Therefore, the conditional density of the vector $(\eg,\ed,X)$ given
$\zeta_\text {max}=h$ is:
\[
f_{E,X}^{\zeta_\text{max}=h}(\eeg,\eed,x)=\frac{1}{2}(2\theta+c_\theta(h))^3
\expp{-(2\theta+c_\theta(h))(\eeg+\eed)}\ind_{\{\eeg\ge 
  0, \, \eed\ge 0, \, -\eeg\le x\le \eed\}}.
\]
For any nonnegative measurable function $\varphi$, we have:
\begin{multline*}
\E^{(\theta)}[ \varphi(\eg+X_-,|X|,\ed-X_+)\ind_{\{X\geq 0\}}
\bigm|\zeta_\text{max}=h] \\    
\begin{aligned}
& =\E^{(\theta)}[\varphi(\eg,X,\ed-X)\ind_{\{X\ge
  0\}}\bigm|\zeta_\text{max}=h]\\ 
%+\E^{(\theta)}[\varphi(\eg+X,-X,\ed)\ind_{X< 0}\bigm|\zeta_{max}=h]\\
& =\int
\varphi(\eeg,x,\eed-x)\frac{1}{2}(2\theta+c_\theta(h))^3\expp{-(2\theta+c_\theta(h)) 
(\eeg+\eed)}\ind_{\{\eeg\ge 0, \, \eed\geq x\geq 0\}}\,d\eeg\, d\eed\,
dx\\ 
%& \qquad +\iiint \varphi(e_g+x,-x,e_d)\frac{1}{2}(2\theta+c_\theta(h))^3\expp{-(2\theta+c_\theta(h))(e_g+e_d)}\ind_{e_d\ge 0}\ind_{-e_g\le x<0}\,de_g\, de_d\, dx\\
& =\int
\varphi(e_1,e_2,e_3)\frac{1}{2}(2\theta+c_\theta(h))^3\expp{-(2\theta+c_\theta(h))(e_1+e_2+e_3)}\ind_{\{e_1\ge
  0, \, e_2\ge 0, \, e_3\ge 0\}}\, de_1\,de_2\, de_3,
%& \qquad +\iiint\varphi(e_1,e_2,e_3)\frac{1}{2}(2\theta+c_\theta(h))^3\expp{-(2\theta+c_\theta(h))(e_1+e_2+e_3)}\ind_{e_1\ge 0}\ind_{e_2\ge 0}\ind_{e_3\ge 0}\, de_1\,de_2\, de_3
\end{aligned}
\end{multline*}
using an  obvious change of variables. Similarly, we get:
\begin{multline*}
\E^{(\theta)}[ \varphi(\eg+X_-,|X|,\ed-X_+)\ind_{\{X<0\}}
\bigm|\zeta_\text{max}=h]\\
=\E^{(\theta)}[ \varphi(\eg+X_-,|X|,\ed-X_+)\ind_{\{X\geq 0\}}
\bigm|\zeta_\text{max}=h].
\end{multline*}
This proves the lemma.
\end{proof}

\section{Renormalized total length of the genealogical tree}
\label{sec:total-l}

Let  $\cf^{(\theta)}=((h_i,\tau_i),  \,  i\in  I)$  be  a  Brownian  forest  under
$\P^{(\theta)}$ with $\theta>0$. Recall that  the tree $\cf^{(\theta)}_{(-\infty ,  0]}$ belongs to
$\T_1$.  
For a forest $\bff\in\T_1$, recall that $\cz _h(\bff)$ denotes the set of vertices of $\cf^{(\theta)}_{(-\infty,0]}$ at level $h$. We shall
also                                                            consider
$\cz_h^*(\bff)=\cz_h(\bff)  \bigcap  \cs(\bff_{(-\infty  ,  h]})^c$  the
extant population at  time $h$ except the point on  the semi-infinite branch
$(-\infty       ,       h]$.        For     $r\leq       h$, we define the set of ancestors
at time $r$ in the past of  the extant population at time $h$ forgetting
the individual in the infinite spine:
\begin{equation}
   \label{eq:defcMst}
\cm_r^h(\bff)=\cg_h(\bff)\bigcap \cz_{r}^*(\bff)
 \end{equation} 
 and its cardinality
\begin{equation}
   \label{eq:defMst}
M_r^h(\bff)=\Card(\cm_r^h(\bff)).
 \end{equation} 
   We also define the  time to the
MRCA of $\cz_t(\cf^{(\theta)})$ as
\begin{equation}
   \label{eq:def-Ah2}
A_t=t-\sup\left\{r\leq t;\, M_r^t =0\right\}. 
\end{equation}

We want  to define the length  of the genealogical tree  $\cg_t(\cf^{(\theta)})$ of
all  extant  individuals  at  time  $t$ (which  is  a.s.~infinite)  by
approximating this  genealogical tree  by trees  with finite  length and
take compensated limits.   Without loss of generality we  can take $t=0$
(since  the distribution  of the  Brownian forest  is invariant  by time
translation).\medskip

Two approximations may be considered here.  The first one is to consider
for  $\varepsilon>0$  the  genealogical  tree  of  individuals  at  time
$t-\varepsilon$,  with descendants  at time  $t$, and  let $\varepsilon$
goes down to 0.  We define the  total length of the genealogical tree of
the current population up to $\varepsilon>0$ in the past as:
\begin{equation}\label{eq:L'eps}
L_\varepsilon=\int_\varepsilon^{\infty } M_{-s}^0 \, ds.
\end{equation}
Set $L=(L_\varepsilon, \varepsilon>0)$. According to \cite{bd:tl}, we have $\E[L_\varepsilon|Z_0]= -
Z_0\log(2\beta \theta \varepsilon)/\beta+O(\varepsilon)$ (see also
\reff{eq:Le1} as $\tilde L_\varepsilon$ is distributed as $L_\varepsilon$), and 
 that the sequence $(L_\varepsilon -
\E[L_\varepsilon|Z_0], \varepsilon>0)$ converges a.s.~as $\varepsilon$ goes
down to zero towards a limit say $\cl$. We recall \reff{eq:def-Z-f}: 
\[
\E\left[\expp{-\lambda \cl}|Z_0\right]
=\expp{2\theta Z_0 \, \varphi(\lambda/(2\beta\theta))}, 
\quad\text{with}\quad
\varphi(\lambda )=\lambda \int_0^1 \frac{1- v^\lambda}{1-v} \, dv
\quad\text{ for all } \lambda>0.
\]

The  second  approximation  consists  in looking  at  the  genealogical  tree
associated with  $n$ individuals picked  at random in the  population at
time $0$.   Recall Definition \eqref{eq:rZdef} of $\rZ_h$.
Let  $(X_k, k\in \N^*)$ be, conditionally
on           $\cf^{(\theta)}$,      independent random variables with distribution 
$\rZ_0(dx)/Z_0$. This models 
individuals uniformly
chosen among the  population living at time $0$.   Define
the ancestors of $X_1, \ldots, X_n$ at time $s< 0$ as:
\[
\cm_s^{(n)}(\cf^{(\theta)})=\{x\in \cm_s^0(\cf^{(\theta)}); \, x\prec X_i \text{ for some } 1\leq i\leq n\}, 
\]
and $M_s^{(n)}=\Card(\cm_s^{(n)}(\cf^{(\theta)}))$ its cardinality. 
We define the total length of the genealogical tree of $n$ individuals
uniformly chosen in the current
population as:
\begin{equation}\label{eq:Lambda'n}
\Lambda_n=\int_0^{\infty } M_{-s}^{(n)} \, ds.
\end{equation}
Set $\Lambda=(\Lambda_n, n\in \N^*)$. The next theorem states that the
two approximations give the same  a.s.~limit.

\begin{theo}
   \label{thm:cvLn}
   The                                                          sequence
   $\left(\Lambda_n - \E[\Lambda_n|Z_0],  n\in \N ^*\right)$ converges
   a.s.~and in $L^2$ towards $\cl$ as $n$ tends to $+\infty$. And we
   have
   $\E[\Lambda_n|Z_0]=\frac{Z_0}{\beta}  \log\left(\frac{n}{2\theta Z_0
     }\right) +R_n$, with $R_n= O(n^{-1}\log(n))$  and $\E[|R_n|]=O(n^{-1}\log(n))$.
\end{theo}

The rest of the section is devoted to the proof of this theorem.

\subsection{Preliminary results}
\label{sec:settingRF}

Let $\eg$ and $\ed$ be two independent exponential random variable with
parameter $2\theta$. Let $\cn=\sum_{i\in I}
\delta_{z_i,\tau_i}$ be, conditionally given $(\eg, \ed)$, distributed as 
a Poisson point measure with intensity $\ind_{[-\eg, \ed]}(z) \, dz
\N^{(\theta)}[d\tau]$. 
We define  $\tilde L=(\tilde L_\varepsilon, \varepsilon>0)$ with:
\[
\tilde L_\varepsilon=\sum_{i\in I} (\zeta_i -\varepsilon)_+,
\]
where $\zeta_i=H(\tau_i)$ is the height of $\tau_i$. 
Let  $(U_k,  k\in \N^*)$  be  independent  random variables  uniformly distributed  on
$[0, 1]$ and  independent of $(\cn, \eg, \ed)$. We set
$ X_{0}=0$,       and      $ X_k=(\eg+\ed)      U_k      -       \eg$ for $k\in
\N^*$.  Fix  $n\in \N^*$.     Let
$ X_{(0,n)}\leq  \cdots  \leq   X_{(n,n)}$   be  the  corresponding  order
statistic  of  $( X_0,  \ldots,   X_n)$. We  set  $ X_{(-1,  n)}=-\eg$  and
$ X_{(n+1,       n)}=\ed$.       We       define       the       interval
$I_{k,n}=( X_{(k-1,      n)},     X_{(k,n)})$     and      its      length
$\Delta_{k,n}= X_{(k,n)}  - X_{(k-1, n)}$  for  $0\leq k\leq  n+1$.  We  set
$\Delta_n=(\Delta_{k,n},  0\leq k\leq  n+1)$.  For $1\leq  k\leq n$,  we
define $\tilde \Lambda=(\tilde \Lambda_n, n\in \N ^*)$ by:
\[
\tilde \Lambda_n=\sum_{k=1}^n \zeta_{k,n}^*.
\quad\text{with}\quad
\zeta_{k,n}^* =\max\{\zeta_i; z_i\in I_{k,n}\}.
\]

Recall     the     definitions     of     $Z_0$     in     \eqref{eq:rZdef},
$L=(L_\varepsilon,    \varepsilon>0)    $   in    \reff{eq:L'eps}    and
$\Lambda=(\Lambda_n,  n\in  \N^*)$   in  \reff{eq:Lambda'n}.  Thanks  to
Proposition \ref{prop:ancestral},  we deduce  that  $(Z_0,  L, \Lambda)$  is
distributed as $(\eg+ \ed, \tilde  L, \tilde \Lambda)$.  So to prove Theorem \ref{thm:cvLn}, it is
enough to
prove the statement  with $\tilde \Lambda$
instead of $\Lambda$.\medskip

For convenience,
we set $Z_0=\eg+\ed$.
Elementary computations give the following lemma. 
Recall that $z_+=\max(z, 0)$. 
\begin{lem}
   \label{lem:moment-zeta}
Let $\theta>0$ and   $\varepsilon>0$. We have:
\begin{equation}
   \label{eq:moment1-zeta}
\N^{(\theta)} [(\zeta-\varepsilon)_+]
=\int_\varepsilon^\infty  c_\theta(h)\, dh=  - \inv{\beta} \log(2\beta\theta
\varepsilon)  + O(\varepsilon),
\end{equation}
\begin{equation}
   \label{eq:moment2-zeta}
\N^{(\theta)} [(\zeta-\varepsilon)_+^2]
=2\int_\varepsilon^\infty   hc_\theta(h)\, dh - 2\varepsilon
\int_\varepsilon^\infty  c_\theta(h)\, dh
= 2\int_0^\infty   hc_\theta(h)\, dh  + O(\varepsilon\log(\varepsilon)). 
\end{equation}
\end{lem}

We deduce  that:
\begin{align}
   \label{eq:Le1}
\E[\tilde L_\varepsilon|Z_0]
&=-\frac{Z_0}{\beta} \log(2\beta\theta
\varepsilon)  + O(\varepsilon),\\
   \label{eq:Le2}
\E[\tilde L_\varepsilon^2|Z_0]
&=2Z_0 \int_0^\infty   hc_\theta(h)\, dh +
\E[\tilde L_\varepsilon|Z_0]^2+    O(\varepsilon\log(\varepsilon)),
\end{align}
where we used that if $\sum_{i\in I} \delta_{x_i}$ is a Poisson point measure with
intensity $\mu(dx)$, then:
\begin{equation}
   \label{eq:PPM-2}
\E\left[\left(\sum_{i\in I}
    f(x_i)\right)^2\right] = \mu(f^2)+ \mu(f)^2. 
 \end{equation} 
Eventually, let us notice that with the change of variable  $u= c_\theta(h) $ (so that
$dh= du / \beta u(u+2\theta)$), we have:
\begin{equation}
   \label{eq:int-hch}
2\int_0^\infty  hc_\theta(h)\, dh= \inv{\beta^2\theta} \int_0^\infty
\frac{\log(v+1)}{v(v+1)} \, dv.
\end{equation}

Recall the definition of $\zeta^*_\delta$ for $\delta>0$, see
\reff{eq:dist-zd}. 
Let $\gamma$ be the Euler constant, and thus:
\[
\gamma=-\int_0^{+\infty } \log(u) \expp{-u}\, du.
\]
We have the following lemma.
\begin{lem}
   \label{lem:zeta*}
Let $\delta>0$. We have:
\begin{equation}
   \label{eq:z*1}
\E[\zeta^*_\delta]=-\frac{\delta}{\beta}\log(2\theta
\delta)+\frac{\delta}{\beta}(1-\gamma)+\frac{\delta}{\beta}
g_1(2\theta\delta),   
\end{equation}
with $|g_1(x)|\leq x (|\log(x)|+2)$ for $x>0$  and
\begin{equation}
   \label{eq:z*2}
\E[(\zeta^*_\delta)^2]=2 \delta  \int_0^\infty hc_\theta(h)\, dh 
+\frac{\delta}{\beta^2\theta} g_2(2\theta \delta), 
\end{equation}
with $|g_2(x)|\leq x (|\log(x)|+2)$ for $x>0$. We also have:
\begin{equation}
   \label{eq:zz}
\E\left[\zeta^*_\delta \sum_{i\in I} (\zeta_i-\varepsilon)_+\right]=2 \delta  \int_0^\infty hc_\theta(h)\, dh 
+ g_3(\delta)
\end{equation}
and there exists a finite constant $c$ such that for all $x>0$ and
$\varepsilon\in (0, 1]$, we have $|g_3(x)|\leq c x^{2}
(1+x)(|\log(x)|+1)(|\log(\varepsilon)|+1) +  
c\varepsilon x(|\log(x)|+1)(1+x)+
\varepsilon^2$. 
\end{lem}

The end of this section is devoted to the proof of Lemma
\ref{lem:zeta*}. 
\subsubsection{Proof of \reff{eq:z*1}}
Using \reff{eq:dist-zd}, we get:
\begin{equation}
   \label{eq:Ezeta*}
\E[\zeta^*_\delta]=\int_0^\infty \P(\zeta^*_\delta>h) \, dh=
\int_0^\infty (1-\expp{-\delta c_\theta(h)})\, dh
=\frac{\delta}{\beta} \int_0^\infty (1-\expp{-u}) \,
\frac{du}{u(u+2\theta \delta)},
\end{equation}
where we used the change of variable  $u=\delta c_\theta(h) $. 
It is easy to check that  for $a>0$:
\begin{equation}
   \label{eq:majo-loga}
\log(1+a)\leq |\log(a)|+\log(2).
\end{equation}
Let $a>0$. We have:
\begin{align*}
\int_0^1 (1-\expp{-u})\,\frac{du}{u(u+a)}
&=\int_0^1 (1-u-\expp{-u})\,\frac{du}{u(u+a)}+ \log(1+a) - \log(a)\\
&=\int_0^1 (1-u-\expp{-u})\,\frac{du}{u^2}+ \log(1+a) - \log(a)
+ ag_{1,0}(a),
\end{align*}
with 
\[
g_{1,0}(a)= -\int_0^1 (1-u-\expp{-u})\,\frac{du}{u^2(u+a)}
\leq  \int_0^1 \frac{du}{2(u+a)}
= \inv{2}(\log(1+a) -\log(a))
\leq   |\log(a)|+ \inv{2}
\]
and $g_{1,0}(a)\geq 0$, where we used that $0\leq  -(1-u-\expp{-u})\leq
u^2/2$ for $u\geq 0$.  We also have:
\begin{align*}
\int_1^\infty  (1-\expp{-u})\,\frac{du}{u(u+a)}
&=\int_1^\infty  (1-\expp{-u})\,\frac{du}{u^2}
- ag_ {1,1}(a),
\end{align*}
with 
\[
g_{1,1}(a)= \int_1^\infty  (1-\expp{-u})\,\frac{du}{u^2(u+a)}
\leq  \int_1^\infty  \frac{du}{u^3}
\leq  \inv{2}\cdot
\]
Notice that, by integration by parts, we have:
\[ 
\int_0^1  (1-u-\expp{-u})\,\frac{du}{u^2}
+\int_1^\infty\!\!  (1-\expp{-u})\,\frac{du}{u^2}
= \expp{-1} +\int_0^1 \!\! \log(u) \expp{-u}du
+1- \expp{-1} +\int_1^\infty \!\!  \log(u) \expp{-u}du=1-\gamma.
\]
We deduce that:
\[
\int_0^\infty  (1-\expp{-u})\,\frac{du}{u(u+a)}
=1-\gamma-\log(a) +g_{1}(a)
\]
with $g_1(a)=\log(1+a) +a g_{1,0}(a)-ag_{1,1}(a)$ and
\[
|g_{1}(a)|
=|\log(1+a) +a g_{1,0}(a)-ag_{1,1}(a)|\leq a (|\log(a)|+2).
\]
Then, use \reff{eq:Ezeta*} to get \reff{eq:z*1}. 

\subsubsection{Proof of \reff{eq:z*2}}
Using \reff{eq:dist-zd}, we get:
\begin{equation}
   \label{eq:Ezeta*2}
\E[(\zeta^*_\delta)^2]
=2\int_0^\infty h(1-\expp{-\delta c_\theta(h)})\, dh
=2 \frac{\delta}{\beta} \int_0^\infty \inv{2\beta\theta}
\log\left(\frac{u+2\theta\delta }{u}\right)(1-\expp{-u}) \,
\frac{du}{u(u+2\theta \delta)},
\end{equation}
where we used the change of variable  $u=\delta c_\theta(h) $. 
Let $a>0$. We set:
\[
g_{2,1}(a)= \int_1^\infty 
\log\left(\frac{u+a}{u}\right)(1-\expp{-u}) \,
\frac{du}{u(u+a)}\cdot
\]
We have using that $0\leq \log(1+x)\leq x$ for $x>0$:
\[
|g_{2,1}(a)|\leq  a \int_1^\infty  \frac{du}{u^3} \leq \frac{ a}{2}\cdot
\]
We also have:
\begin{align*}
 \int_0^1
\log\left(\frac{u+a}{u}\right)(1-\expp{-u}) \,
\frac{du}{u(u+a)}
&= \int_0^1
\log\left(\frac{u+a}{u}\right) \,
\frac{du}{u+a}+g_{2,2}(u)\\
&=\int_0^\infty \frac{\log(v+1)}{v(v+1)} \, dv
-g_{2,3}(a)+g_{2,2}(a),   
\end{align*}
with the change of variable $v=a/u$ as well as:
\[
g_{2,2}(a)= \int_0^1
\log\left(\frac{u+a}{u}\right)(1-u-\expp{-u}) \,
\frac{du}{u(u+a)} 
\quad\text{and}\quad 
g_{2,3}(a)=\int_0^{a}
\frac{\log\left(v+1\right)}
{v(v+1)} \, dv.
\]
We have, using $\log(1+v)\leq v$ for $v>0$ (twice), that:
\[
0\le g_{2,3}(a)\leq   \int_0^{a}
\frac{dv}{v+1}
\leq a.
\]
We have, using $|1-u-\expp{-u}|\leq u^2/2$ if $u>0$ for the first
inequality and  \reff{eq:majo-loga} for the last, that:
\[
|g_{2,2}(a)|\leq  \inv{2} \int_0^1 \log\left(1+ \frac{a}{u}\right) \,
\frac{udu}{(u+a)} 
\leq  \frac{a}{2} \int_0^1 
\frac{du}{(u+a)} 
\leq a( |\log(a)|+\inv{2}).
\]
We deduce that:
\[
 \int_0^\infty 
\log\left(\frac{u+a}{u}\right)(1-\expp{-u}) \,
\frac{du}{u(u+a)}
=\int_0^\infty \frac{\log(v+1)}{v(v+1)} \, dv
+ g_2(a)
\]
and
\[
|g_2(a)|=|g_{2,1}(a)-g_{2,3}(a)+g_{2,2}(a)|\leq   a( |\log(a)|+2).
\]
Then, use \reff{eq:Ezeta*2} as well as the identity \reff{eq:int-hch}  to get \reff{eq:z*2}. 

\subsubsection{Proof of \reff{eq:zz}}
Using properties of Poisson point measures, we get that if $\sum_{j\in
  J}\delta _{\zeta_j}$ is a Poisson point measure with intensity
$\delta\N[d\zeta]$ and $\zeta^*_\delta=\max_{j\in J} \zeta_j$, then for
any measurable non-negative functions $f$ and $g$, we have:
\[
\E\left[f(\zeta^*_\delta)\expp{-\sum_{j\in J} g(\zeta_j)}\right]
= \E\left[f(\zeta^*_\delta)\expp{- g(\zeta_\delta^*) -
    G(\zeta_\delta^*)}\right]
\quad\text{with}\quad 
G(r)=\delta \N\left[(1-\expp{-g(\zeta)})\ind_{\{\zeta<r\}}\right].
\]
We deduce that: 
\[
\E\left[\zeta^*_\delta \sum_{i\in I} (\zeta_i-\varepsilon)_+\right]=
\E[\zeta^*_\delta(\zeta^*_\delta-\varepsilon)_+]+ \delta
g_{3,1}(\delta),
\]
with $g_{3,1}(\delta)= 
\E\left[\zeta^*_\delta 
  \N\left[(\zeta-\varepsilon_+)
  \ind_{\{\zeta<h\}}\right]_{|h=\zeta^*_\delta}\right]$. 
According to \reff{eq:z*1}, there exists a finite constant $c>0$ such
that for all $\delta>0$, we have $
\E[\zeta^*_\delta]\leq  c \delta(|\log(\delta)|+1)(1+\delta)$.
We deduce from \reff{eq:moment1-zeta} that there exists a finite constant $c$ independent of $\delta>0$ and
$\varepsilon\in (0,1]$ such that: 
\[
g_{3,1}(\delta) \leq  \E[\zeta^*_\delta] \N[(\zeta-\varepsilon)_+]
\leq c \delta(|\log(\delta)|+1)(1+\delta)(|\log(\varepsilon)|+1).
\]
We also have:
\[
\E[\zeta^*_\delta(\zeta^*_\delta-\varepsilon)_+]=
\E[(\zeta^*_\delta)^2]-
\E[(\zeta^*_\delta)^2\ind_{\{\zeta^*_\delta<\varepsilon\}} ]
-\varepsilon \E[\zeta_\delta^*\ind_{\{\zeta^*_\delta>\varepsilon\}}]= 2
\delta  \int_0^\infty hc_\theta(h)\, dh + g_{3,2}(\varepsilon,\delta), 
\]
with, thanks to \reff{eq:z*1}
and \reff{eq:z*2}, 
 $|g_{3,2}(\varepsilon,\delta)|\leq
c\delta^2(|\log(\delta)|+1)+\varepsilon^2+ c\varepsilon \delta(|\log(\delta)|+1)(1+\delta)$, for some finite constant $c$ independent of
$\delta>0$ and $\varepsilon>0$. 
We deduce that:
\[
\E\left[\zeta^*_\delta \sum_{i\in I} (\zeta_i-\varepsilon)_+\right]=
2 \delta  \int_0^\infty hc_\theta(h)\, dh 
+ g_3(\delta)
\]
and for some finite constant $c$ independent of
$\delta>0$ and $\varepsilon\in (0,1]$. 
\[
|g_3(\delta)|\leq c \delta^{2}
(1+\delta)(|\log(\delta)|+1)(|\log(\varepsilon)|+1) +  
c\varepsilon \delta(|\log(\delta)|+1)(1+\delta)+
\varepsilon^2 .
\]

\subsection{A technical lemma}
An elementary induction gives for $n\in \N$ that:
\[
\int_0^1 (1-x)^n |\log(x)|\, dx= \frac{H_{n+1}}{n+1} 
\quad\text{and}\quad 
\int_0^1 (1-x)^n \log^2(x)\, dx= \frac{2}{n+1} \sum_{k=1}^{n+1}
\frac{H_{k}}{k}  ,
\]
where $H_n=\sum_{k=1}^n k^{-1}$ is the harmonic sum. 
Recall that $H_{n}=\log(n) +\gamma + (2n)^{-1} +O(n^{-2})$. So we deduce that:
\begin{equation}
   \label{eq:EUlog1}
(n+1) \int_0^1 (1-x)^n |\log(x)|\, dx= \log(n)+\gamma +\frac{3}{2n}+ O(n^{-2}).
\end{equation}
It is also easy to deduce that for $a,b\in \{1,2\}$:
\begin{equation}
   \label{eq:EUlog}
\int_0^1 x^a(1-x)^n |\log(x)|^b \, dx=
O\left(\frac{\log^b(n)}{n^{a+1}}\right). 
\end{equation}

Recall $\tilde \Lambda_n$ and $\Delta_n$ defined in Section
\ref{sec:settingRF}. 
We give a technical  lemma. In this lemma $O(f(n))$ denotes a function,
say $\phi$,
of $Z_0$ and $n$ such that $|\phi(Z_0, n)|\leq  Q(Z_0) f(n)$ for some
positive function $Q$ such that  $Q(Z_0)$ is  integrable. The explicit
function $Q$ is unimportant and thus not specified. 
\begin{lem}
   \label{lem:ELn}
We have:
\begin{equation}
   \label{eq:ELnD1}
\E[\tilde \Lambda_n|\Delta_n]=\frac{Z_0}{\beta}(1-\gamma) - \sum_{k=1}^ {n}
\frac{\Delta_{k,n}}{\beta}\log(2\theta \Delta_{k,n}) + W_n,
\end{equation}
with 
$\E[|W_n|\, |Z_0]=O(n^{-1}\log(n))$   and 
\begin{equation}
   \label{eq:ELnZ1}
\E[\tilde \Lambda_n|Z_0]= 
\frac{Z_0}{\beta}\log\left(\frac{n}{2\theta Z_0}\right)+ O(n^{-1}\log(n)).
\end{equation}
We have also:
\begin{equation}
   \label{eq:ELnD2}
\E[\tilde \Lambda_n^2|Z_0]= 
2 Z_0 \int_0^\infty  hc_\theta(h)\, dh + \E[\tilde \Lambda_n\,|\, Z_0]^2+
O(n^{-1}\log^2(n)).
\end{equation}
\end{lem}

\begin{proof}
  We       first       prove       \reff{eq:ELnD1}.        We       have
  $\E[\tilde \Lambda_n|\Delta_n]=                                    \sum_{k=1}^n
  \E[\zeta^*_\delta]_{|\delta=\Delta_{k,n}}$.
  We deduce from \reff{eq:z*1} that \reff{eq:ELnD1} holds with:
\[
W_n=\frac{\Delta_{0,n}+\Delta_{n+1,n}}{\beta} (\gamma-1) + \inv{\beta}
\sum_{k=1}^n \Delta_{k,n} g_1( 2\theta \Delta_{k,n}).
\]
Since, conditionally on $Z_0$, the random variables $\Delta_{k,n}$ are
all distributed as $Z_0 \tilde U_n$, where $\tilde U_n$ is independent
of $Z_0$ and has distribution
$\beta(1,n+1)$, we deduce using \reff{eq:EUlog} that:
\[
\E[|W_n|\, |Z_0]\leq 2\frac{(1-\gamma)Z_0}{\beta}\E[\tilde U_n] + 
n\frac{2\theta Z_0^2}{\beta} \E[\tilde U_n^2( |\log(2\theta Z_0 \tilde
U_n)|+ 2)|Z_0]= O(n^{-1}\log(n)).
\]

We then prove \reff{eq:ELnZ1}. 
Taking the expectation in  \reff{eq:ELnD1} conditionally on $Z_0$, we
get:
\[
\E[\tilde \Lambda_n|Z_0]= \frac{Z_0}{\beta}(1-\gamma) -n\frac{Z_0}{\beta}
\ch(2\theta Z_0)
+\E[W_n|Z_0],
\]
where 
\begin{equation}
   \label{eq:def-ch}
\ch(a)= \E[\tilde U_n\log(a \tilde U_n)]. 
\end{equation}
We deduce from \reff{eq:EUlog1} that:
\begin{equation}
   \label{eq:DLcha}
n\ch(a)=\log(a) - \log(n)+ 1 -\gamma + O(n^{-1}\log(n)). 
\end{equation}
This gives:
\[
\E[\tilde \Lambda_n|Z_0]= 
\frac{Z_0}{\beta}\log\left(\frac{n}{2\theta Z_0}\right) + O(n^{-1}\log(n)).
\]

We finally prove \reff{eq:ELnD2}. We have:
\begin{equation}
   \label{eq:L2D}
\E\left[\tilde \Lambda_n^2|\Delta_n\right]
=\sum_{k=1}^n \E\left[(\zeta_\delta^*)^2\right]_{|\delta=\Delta_{k,n}}
- \sum_{k=1}^n \E\left[\zeta_\delta^*\right]^2_{|\delta=\Delta_{k,n}}
+ \E\left[\tilde \Lambda_n|\Delta_n\right]^2.
\end{equation}
We have thanks to \reff{eq:z*2}: 
\[
\sum_{k=1}^n \E\left[(\zeta_\delta^*)^2\right]_{|\delta=\Delta_{k,n}}
= 2 Z_0  \int_0^\infty  hc_\theta(h)\, dh 
+ W_{1,n},
\]
with 
\[
W_{1,n}=- 2(\Delta_{0, n}+\Delta_{n+1,n})
\int_0^\infty  hc_\theta(h)\, dh 
+ 
\sum_{k=1}^n \frac{\Delta_{k,n}}{\beta^2\theta} g_2(2\theta
\Delta_{k,n}).
\]
Using similar computations as the ones used to bound $\E[|W_n|\, |\,
Z_0]$, we get $  \E[|W_{1,n}|\, |Z_0]= O(n^{-1} \log(n))$ so that
\[
\E\left[\sum_{k=1}^n
  \E\left[(\zeta_\delta^*)^2\right]_{|\delta=\Delta_{k,n}}
\, |\, Z_0\right]
= 2 Z_0  \int_0^\infty  hc_\theta(h)\, dh 
+ O(n^{-1} \log(n)).
\]

Thanks to \reff{eq:z*1}, we have $\E\left[\zeta_\delta^*\right]^2\leq  c
\delta^2 
(|\log(\delta)|+1)^2 (1+\delta)^2$
for some finite constant $c$ which does not depend on $\delta$. 
We set $\ch_2(a)=\E\left[\tilde U_n^2 \log^2(a \tilde U_n)(1+\tilde
  U_n)^2\right]$, and using \reff{eq:EUlog}, we get:
\begin{equation}
   \label{eq:ch2}
\ch_2(a)=O(n^{-3}\log^2(n)) =O(n^{-2}\log^2(n)).
\end{equation}
We
deduce that:
\[
\E\left[\sum_{k=1}^n \E\left[\zeta_\delta^*\right]^2_{|\delta=\Delta_{k,n}}
\, | \, Z_0\right] = O(n^{-1} \log^2(n)).
\] 
Then using \reff{eq:ELnZ1}, elementary computations give:
\[
\E\left[\E\left[\tilde \Lambda_n|\Delta_n\right]^2\,|Z_0\right]
=2 \frac{Z_0}{\beta}(1-\gamma) \E[\tilde \Lambda_n|Z_0] - \frac{Z_0^2}{\beta^2}
(1-\gamma)^2 + \inv{\beta^2}J_{1,n}+ J_{2,n}-\frac{2}{\beta}J_{3,n},
\]
with $J_{2,n}=\E[W_n^2|Z_0]$, 
\[
J_{1,n}=\E\left[\left(\sum_{k=1}^n \Delta_{k,n}
  \log(2\theta \Delta_{k,n})\right)^2\, \Big|\, Z_0\right]
\quad\text{and}\quad
J_{3,n}= \E\left[W_n\left(\sum_{k=1}^n \Delta_{k,n}
  \log(2 \theta \Delta_{k,n})\right)\, \Big|\, Z_0 \right].
\]
By  Cauchy-Schwartz,  we  have $|J_{3,n}|\leq  \sqrt{J_{1,n}J_{2,n}}$.  
Using $(\sum_{k=1}^n a_k)^2\leq  n \sum_{k=1}^n a_k^2$, we also get:
\[
J_{2,n}\leq  \frac{8}{\beta^2} (\gamma-1)^2 Z_0^2\E[\tilde U_n^2] + 
\frac{2n}{\beta^2} Z_0^2\E\left[\tilde U_n^2 g_1^2(2\theta Z_0 \tilde
  U_n)\right]  = O(n^{-2}).
\]
By
independence, we obtain:
\[
J_{1,n}=n(n-1) \E\left[\Delta_{1,n} \log(2\theta \Delta_{1,n})|Z_0\right]^2
+ n \E\left[\Delta_{1,n}^2 \log^2(2\theta \Delta_{1,n})|Z_0\right].
\]
Recall the function $\ch$ defined in \reff{eq:def-ch} and its asymptotic expansion \reff{eq:DLcha}. We have, using \reff{eq:ch2}, that:
\[
J_{1, n}= n(n-1) Z_0^2 \ch(2\theta Z_0)^2 + nZ_0^2 \ch_2(2 Z_0)
= Z_0^2 \left(-\log\left(\frac{n}{2\theta Z_0} \right) + 1-\gamma\right)^2+
O(n^{-1}\log^2(n)). 
\]
So we deduce that:
\begin{align*}
\inv{\beta^2}J_{1,n}+ J_{2,n}-\frac{2}{\beta}J_{3,n}
&=  \left(-\frac{Z_0}{\beta}\log\left(\frac{n}{2\theta Z_0}\right)+
  \frac{Z_0}{\beta}(1-\gamma)\right)^2 + O(n^{-1}\log^2(n))\\
&= \left(-\E[\tilde \Lambda_n|Z_0] + \frac{Z_0}{\beta}(1-\gamma)\right)   ^2 +
  O(n^{-1}\log^2(n)). 
\end{align*}
We deduce that:
\[
\E\left[\E\left[\tilde \Lambda_n|\Delta_n\right]^2\,|Z_0\right]
=\E[\tilde \Lambda_n\,|\, Z_0]^2+   O(n^{-1}\log^2(n) ). 
\]
So in the end, using \reff{eq:L2D}, we get:
\[
\E\left[\tilde \Lambda_n^2\,|\, Z_0\right]
=2 Z_0 \int_0^\infty  hc_\theta(h)\, dh + \E[\tilde \Lambda_n\,|\, Z_0]^2+
O(n^{-1}\log^2 (n)). 
\]
\end{proof}

\subsection{Proof of Theorem \ref{thm:cvLn}}
\label{proof-prop}
We shall keep  notations from Section \ref{sec:settingRF}.
We set  $J_n(\varepsilon)=\E\left[\left(\tilde \Lambda_n-
  \tilde   L_\varepsilon\right)^2 |Z_0\right]$.
We have:
\[
J_n(\varepsilon)=\E[\tilde \Lambda_n^2|Z_0] + \E[\tilde L_\varepsilon^2|Z_0] -
2\E[\tilde \Lambda_n \tilde L_\varepsilon|Z_0]. 
\]
By conditioning with respect to $\Delta_n$, and using the independence,
we get: 
\[
\E[\tilde \Lambda_n \tilde  L_\varepsilon|Z_0]
=\E\left[\E[\tilde \Lambda_n \tilde L_\varepsilon|\Delta_n]|Z_0\right]
=\Sigma_n + \E\left[\E[\tilde \Lambda_n|\Delta_n]\E[
  \tilde L_\varepsilon|\Delta_n]\, \Big|\, Z_0\right]
=\Sigma_n+ \E[\tilde \Lambda_n|Z_0]\E[ \tilde  L_\varepsilon|Z_0],
\]
where we used that $\E[\tilde L_\varepsilon|\Delta_n]=\E[\tilde
L_\varepsilon|Z_0]$ for 
the last equality, and:
\[
\Sigma_n=
 \E\left[\sum_{k=1}^n \E\left[\zeta_{k,n}^* \sum_{z_i\in I_{k,n}} (\zeta_i
    -\varepsilon)_+ \, \Big|\,\Delta_n\right]
- \sum_{k=1}^n \E[\zeta_{k,n}^*  |\Delta_n]
\E\left[\sum_{z_i\in I_{k,n}} (\zeta_i
    -\varepsilon)_+ \,\Big|\, \Delta_n\right]
\, \Big |\, Z_0\right].
\]
So using \reff{eq:Le2} and \reff{eq:ELnD2}, we get:
\[
J_n(\varepsilon)= 4Z_0 \int_0^\infty   hc_\theta(h)\, dh
 - 2\Sigma_n + \left(\E[\tilde \Lambda_n|Z_0]-  \E[\tilde L_\varepsilon|Z_0]\right)^2+
O(\varepsilon\log(\varepsilon))+  O(n^{-1}\log^2(n)).
\]
Then taking $\varepsilon \asymp n^{-1}$, we get, using \reff{eq:Le1}, \reff{eq:ELnZ1} and Lemma
\ref{lem:Sigma} below: 
\[
J_n(\varepsilon)= \frac{Z_0^2}{\beta^2} \log^2\left(n\varepsilon
  \frac{\beta }{Z_0} \right)+ O (n^{-1}\log^2(n)).
\]
We deduce that $\tilde \Lambda_n - \tilde L_{Z_0/(n\beta)}$ converges in
probability to $0$ and, by  Borel-Cantelli lemma almost surely along the
sub-sequence      $n^3$.        Recall      that       the      sequence
$(\tilde L_\varepsilon -  \E[\tilde L_\varepsilon|Z_0], \varepsilon>0) $
converges a.s., as  $\varepsilon$ goes down to $0$, towards  a limit say
$\tilde                \cl$.                  Notice                that
$\E[\tilde      L_{Z_0/n\beta}|Z_0]=\E[\tilde      \Lambda_n|Z_0]      +
O(n^{-1}\log(n))$
and             thus,             we             deduce             that
$(\tilde  \Lambda_{n^3}-   \E[\tilde  \Lambda_{n^3}|Z_0],   n\in  \N^*)$
converges also a.s.~towards $\tilde \cl$.  Then  use \reff{eq:ELnD1} to
get that for $k\in [n^3, (n+1)^3)$:
\[
\tilde \Lambda_{n^3} - \E[\tilde \Lambda_{n^3}|Z_0] + O(n^{-1}\log(n))
\leq  \tilde \Lambda_k - \E[\tilde \Lambda_k|Z_0] 
\leq  \tilde \Lambda_{(n+1)^3} - \E[\tilde \Lambda_{(n+1)^3}|Z_0] + O(n^{-1}\log(n)).
\]
Then conclude that  $(\tilde \Lambda_{n}-  \E[\tilde \Lambda_{n}|Z_0], n\in \N^*)$
converges   also a.s.~towards $\cl$.

\begin{lem}
   \label{lem:Sigma}
Let $\varepsilon \asymp n^{-1}$. We have:
\[
\Sigma_n=  2 Z_0 \int _0^\infty  hc_\theta(h)\, dh +  O(n^{-1}\log^2(n)).
\]
\end{lem}
\begin{proof}
   We have $\E\left[\sum_{z_i\in I_{k,n}} (\zeta_i
    -\varepsilon)_+ \,\Big|\, \Delta_n\right]= \Delta_{k,n}
  \N[(\zeta-\varepsilon)_+]$. Thanks to \reff{eq:z*1}, \reff{eq:EUlog1} and
  \reff{eq:EUlog}, we get: 
\begin{align*}
  \E\left[\sum_{k=1}^n \Delta_{k,n} \E[\zeta_{k,n}^*  |\Delta_n]
\, \Big |\, Z_0\right] 
&=\frac{nZ_0^2}{\beta} \E\left[\tilde U_n^2 \left(\log(2\theta Z_0 \tilde
  U_n) + (1-\gamma) + g_1(2\theta Z_0 \tilde U_n)\right) \,|Z_0\right]\\
&= O(n^{-2}\log(n)).
\end{align*}
We deduce from \reff{eq:moment1-zeta} with $\varepsilon \asymp n^{-1}$ that:
\[
\E\left[\sum_{k=1}^n \E[\zeta_{k,n}^*  |\Delta_n]
\E\left[\sum_{z_i\in I_{k,n}} (\zeta_i
    -\varepsilon)_+ \,\Big|\, \Delta_n\right]
\, \Big |\, Z_0\right]= O(n^{-1}\log^2(n)).
\]
According to \reff{eq:zz}, we have:
\[
\sum_{k=1}^n \E\left[\zeta_{k,n}^* \sum_{z_i\in I_{k,n}} (\zeta_i
    -\varepsilon)_+ \, \Big|\,\Delta_n\right]
= 2 Z_0 \int _0^\infty  hc_\theta(h)\, dh + W'''_n,
\]
with 
\[
W'''_n=- 2(\Delta_{0,n}+\Delta_{n+1,n})\int _0^\infty  hc_\theta(h)\, dh+
\sum_{k=1}^n g_3(\Delta_{k,n}).
\]
Since $\varepsilon \asymp n^{-1}$, we deduce that 
\[
\E[|W'''_n||Z_0]
\leq \frac{2 Z_0}{n+1} \int _0^\infty  hc_\theta(h)\, dh
+ O(n^{-1}\log^2(n)).
\]
This gives the result. 
\end{proof}

\subsection*{Acknowledgments}
We thank the referees for their valuable comments which improves the
presentation of the paper and broadens  the bibliography.

\bibliographystyle{abbrv}
\bibliography{biblio}

\end{document}